\documentclass[12pt,a4paper, oneside, bold,secthm,seceqn,amsthm,ussrhead,reqno]{article}
\usepackage[utf8]{inputenc}
\usepackage[english]{babel}
\usepackage[symbol]{footmisc}
\usepackage{amssymb,amsmath,amsthm,amsfonts,xcolor,enumerate,hyperref, comment,longtable, cleveref}
\usepackage{verbatim}
\usepackage{times}
\usepackage{cite}
\usepackage{pdflscape}
\usepackage{ulem}
\usepackage[mathcal]{euscript}
\usepackage{tikz}
\usepackage{hyperref}
 
\usepackage{cancel}
\usepackage{stmaryrd}
 
\usepackage{amsfonts}
\usepackage{amssymb}
\usepackage{times}
\usepackage{xcolor}
\usepackage{tkz-graph}
\usepackage{url}
\usepackage{float}
\usepackage{tasks}
\usepackage{array}

\usepackage{cite}
\usepackage{hyperref}

 \usepackage{fancyhdr} 
\usepackage{amsfonts}
\usepackage{amsmath}
\usepackage{eurosym}
\usepackage{geometry}

\usepackage{caption,booktabs}

\theoremstyle{plain}
\newtheorem{theorem}{Theorem}
\newtheorem{lemma}[theorem]{Lemma}
\newtheorem{proposition}[theorem]{Proposition}
\newtheorem{remark}[theorem]{Remark}

\newtheorem{definition}[theorem]{Definition}
\newtheorem{corollary}[theorem]{Corollary}

\usetikzlibrary{arrows}

\usepackage{fouriernc}

\begin{document}
 

\noindent{\Large
$\delta$-Leibniz   algebras and related $\delta$-type algebras} \footnote{The this work is supported by FCT   2023.08031.CEECIND and UID/00212/2025.} 
\footnote{The authors are grateful to Saïd Benayadi and Samuel Lopes for insightful discussions that drew their attention to the topic under consideration; 
and to Abror Khudoyberdiyev for some constructive comments.}

 \bigskip

\begin{center}

 {\bf Jobir Adashev\footnote{Institute of Mathematics, Uzbekistan  Academy of  Sciences, Uzbekistan; \  
 Chirchiq State Pedagogical University,  Tashkent,  Uzbekistan; \  adashevjq@gmail.com} \&
   Ivan Kaygorodov\footnote{CMA-UBI, University of  Beira Interior, Covilh\~{a}, Portugal; \  \ kaygorodov.ivan@gmail.com}  }

\end{center}

\

 \bigskip

\noindent {\bf Abstract:}
{\it  
This paper introduces and investigates the structure of $\delta$-Leibniz algebras, which serve as a parametric generalization of classical Leibniz algebras defined by a scalar $\delta$. The authors define $\delta$-Lie algebras, $\delta$-Lie dialgebras, and $\delta$-Zinbiel algebras via a standard procedure and study their fundamental properties. Furthermore, the research describes symmetric $\delta$-Leibniz algebras and algebras of $\delta$-biderivation type, establishing their connections with nilalgebras. Finally, these results provide a unified framework for understanding various classes of non-associative algebras through the lens of the $\delta$ parameter.
}

 \bigskip 


\bigskip 

 \
 

	 \bigskip


\tableofcontents 
\newpage
\section*{Introduction}

The concept of Leibniz algebras as Lie dialgebras was popularized by Loday in the 1990s, and it is still under active consideration \cite{Leib1,Leib2,Leib3,KM,FKS}.
The principal role in the definition of Leibniz algebras is played by the fact that each right  multiplication on an element from the same vector space gives a derivation of the algebra. The concept of derivation is also playing a central role in Lie, Jordan, Novikov, Poisson, etc. algebras. 
Recently, the idea of anti-Leibniz algebras, via changing the notion of derivations to antiderivations $\big($also known as $(-1)$-derivations$\big)$, 
was introduced by Braiek, Chtioui, and  Mabrouk \cite{antileib}.
The algebraic study of the theory of anti-Leibniz algebras was continued for conformal algebras \cite{CT} and bialgebras \cite{antilb}.
Anti-Leibniz algebras are a non-commutative generalization of Jacobi-Jordan $\big($also known as mock-Lie$\big)$ algebras. 

\medskip 

The present paper is dedicated to a study of $\delta$-Leibniz algebras, 
obtained as a simultaneous generalization of Leibniz and anti-Leibniz algebras via $\delta$-derivations.
The concept of $\delta$-type algebras via $\delta$-derivations first appeared in \cite{dP,DET} in the context of 
simultaneous generalization of Poisson and anti-Poisson algebras.
Later, a similar notion was considered in the context of Novikov algebras \cite{dN,dNalg}.
As it was found, anti-Novikov algebras are very related to anti-pre-Lie algebras introduced by Bai and  Liu \cite{LB},
which were also studied in \cite{CSL, bai25,LB24}.  
The principal identity of  anti-pre-Lie algebras gives a generalization of the antiassociative identity 
$\big($see \cite{FKK}  and  \cite{R22} about antiassociative algebras$\big).$
Antiassociative algebras give a tool for the construction of anti-dendriform algebras \cite{aan25,bai23,SZ}.
The concept of $\delta$-commutator presents a special type of mutation of algebras, and it was deeply studied for Leibniz and Zinbiel algebras by Dzhumadildaev \cite{dzh08,dzh09}.
On the other hand, quasi-associative and quasi-alternative algebras were defined via a suitable $\delta$-commutator product \cite{Ded}.

\medskip  

The present paper is organized as follows. 
Section \ref{dlie} introduces the concept of $\delta$-Lie algebras and shows that almost all $\delta$-Lie algebras will be $2$-step nilpotent:
some interesting cases appear for $\delta=-\frac 12$ $\big($Lemma \ref{ancommass}$\big).$
Section \ref{ddilie} is dedicated to introducing of $\delta$-Lie dialgebras:
they are equivalent to $2$-step nilpotent algebras or  to antiassociative anti-right-commutative algebras $\big($Theorem \ref{ddilieth}$\big).$
The rest of Section  \ref{ddilie} is dedicated to a study of antiassociative anti-right-commutative algebras: the first known non-$2$-step nilpotent antiassociative anti-right-commutative algebras appear in dimension $7$ $\big($Remark \ref{dimaac}$\big);$
obtain a basis of the free  antiassociative anti-right-commutative algebra generated by a set $X$  $\big($Proposition \ref{basaac}$\big);$
and construct the identity  that governs the Koszul dual operad of the operad that governs the  antiassociative anti-right-commutative algebras.  
Section \ref{dl} is the central part of our research and discusses $\delta$-Leibniz algebras.
Namely, subsection \ref{antiL} is dedicated to a particular case of $\delta$-Leibniz algebras:
we prove that there are no simple anti-Leibniz algebras $\big($Theorem \ref{simplantil}$\big);$
and that each symmetric anti-Leibniz algebra  is a central extension of a suitable
Jacobi-Jordan algebra $\big($Theorem \ref{cenrext}$\big).$
Subsection \ref{deltL} discusses some questions related to $\delta$-Leibniz algebras:
we give a classification of $2$-dimensional $\delta$-Leibniz algebras $\big($Theorem \ref{cldn}$\big)$
and study the nilpotency index of $\delta$-Leibniz algebras $\big($Proposition  \ref{nilpL}$\big).$
As a corollary, we have the upper bound for the length of finite-dimensional $\delta$-Leibniz algebras $\big($Corollary \ref{lengthL}$\big).$
We also proved that each $\delta$-associative algebra is $3$-step nilpotent $\big($Proposition \ref{deltaass}$\big)$ and characterized commutative $\delta$-Leibniz algebras in  Theorem \ref{comL}.
Proposition \ref{liadmd} gives identities that characterize Lie-admissible $\delta$-Leibniz algebras
$\big($for the first time, these identities were discovered by Malcev in 1950 \cite{M50}$\big)$.
Theorem \ref{Lcons} states that each $\delta$-Leibniz algebra is a conservative algebra 
$\big($the definition of conservative algebras was gives by Kantor in 1972  \cite{K72}$\big)$.
Subsection \ref{ddd} generalizes some known results about relations between 
Leibniz and anti-Leibniz algebras, associative and antiassociative algebras, and dialgebras in the context of $\delta$-Leibniz algebras.
Subsection \ref{syml} gives a characterization of symmetric $\delta$-Leibniz algebras:
they are power-associative and if $\delta\neq \frac 12$ they are nilalgebras with nilindex $3$
$\big($Theorem \ref{symmetric}$\big).$
Section \ref{dZ} introduces the notion of $\delta$-Zinbiel algebras as Koszul dual to $\delta$-Leibniz algebras $\big($it is a generalization of the way for obtaining Zinbiel algebras$\big).$
Namely, we give a classification of $2$-dimensional $\delta$-Zinbiel algebras $\big($Theorem \ref{cldz}$\big)$ and study the nilpotency index of $\delta$-Zinbiel algebras $\big($Proposition \ref{nilpdZ}$\big).$
As a corollary, we have the upper bound for the length of $\delta$-Zinbiel algebras $\big($Corollary \ref{lengthZ}$\big).$
Proposition \ref{34} also gives identities that characterize Lie-admissible $\delta$-Zinbiel algebras
$\big($as it was mentioned below, we remember that for the first time, these identities were discovered by Malcev \cite{M50}$\big)$.
Subsection \ref{antiz} is dedicated to proving an analog of the most famous results in the theory of Zinbiel algebras. Namely, we prove that each finite-dimensional anti-Zinbiel algebra is nilpotent  $\big($Theorem \ref{naz}$\big).$
The last section gives a way to connect Lie-admissible and Jacobi-Jordan admissible algebras via $\delta$-commutator context: we characterize so-called symmetric $\delta$-Leibniz admissible algebras and prove that they are nilalgebras  if $\delta\notin \big\{ \pm 1,\frac12 \big\}$ $\big($Theorem \ref{nilds}$\big).$
Algebras of $\delta$-biderivation-type, considered in the last subsection, are a particular case of symmetric $\delta$-Leibniz admissible algebras and give a bridge between algebras of biderivation-type 
and algebras of anti-biderivation-type, previously considered by Benayadi and co-authors \cite{BBK,BO24},
due to the concept of $\delta$-biderivations, prevously considered by Wang and  Chen \cite{wc}.

\medskip 

  \noindent  
{\bf Notations.}
We do not provide some well-known definitions
$\big($such as definitions of 
Lie algebras,
Leibniz algebras,   nilpotent algebras, solvable algebras,  etc.$\big)$ 
and refer the readers to consult previously published papers. 
For a set of vectors $S,$ we denote by  
$\big\langle S \big\rangle$ the vector space generated by $S.$ 
For the $\delta$-commutator and $\delta$-associator, we will use the standard notations:

\begin{center}
    $\big[x,y\big]_{\delta} : = x\ast y- \delta y\ast x,$   \ and \ 
    $\big(x,y,z\big)_{\delta}^{\ast}:=(x\ast y) \ast z - \delta x\ast(y\ast z).$ 
    \end{center}
\noindent
In general, we are working with the complex field, but some results are correct for other fields.
We also almost always  assume that $\delta  \neq 1$ and 
algebras  under our consideration are nontrivial,  i.e., they have   nonzero multiplications.
 \section{$\delta$-Lie   algebras}\label{dlie}

\begin{definition}\label{12der}
Let ${\mathcal A}$ be an algebra and $\delta$ be a fixed complex number.  
Then  a linear map $\varphi$ is a $\delta$-derivation if it satisfies
\begin{center}
$\varphi(x y) \ = \ \delta  \big(\varphi(x)y+ x \varphi(y)\big).$
\end{center}
\end{definition}

If $\delta=1$ $\big($resp., $\delta=-1\big),$ we have a derivation $\big($resp., antiderivation\footnote{Let us remember that the notion of antiderivations plays an important role in the definition of mock-Lie algebras\cite{Z17} and (transposed) anti-Poisson algebras \cite{dP}; 
in the characterization of Lie algebras with identities \cite{fant} and so on.
In the anticommutative case,  antiderivations coincide with reverse derivations defined by Herstein in 1957 \cite{her}.}$\big)$.
The notion of $\delta$-derivations was introduced in a paper by Filippov \cite{fil1} (see also references in \cite{k23,aae23,zz}).

\begin{definition}\label{deltaLie}
Let $L$ be an anticommutative algebra and $\delta$ be a fixed complex number. Then $L$ is called a  $\delta$-Lie  algebra if the following identity holds true:
\begin{equation*} 
       (xy)z\ =\  \delta \big((xz)y+x(yz)\big).
\end{equation*}
\end{definition}

Let $\mathcal{L}ie_{\delta}$ be a variety of $\delta$-Lie algebras  for some $\delta$. Then $\mathcal{L}ie_0$ is the variety of anticommutative $2$-step nilpotent algebras, and $\mathcal{L}ie_1$ is the variety of Lie algebras.
Hence, the varieties of $\delta$-Lie  algebras give a type of continuous deformation 
of the variety of Lie algebras to the variety of anticommutative $2$-step nilpotent algebras and vice versa.

\begin{lemma}\label{3} Let $L$ be a $\delta$-Lie algebra  $\big(\delta\neq1\big).$ 
Then $L$ is  antiassociative 
$\big($i.e., $(x,y,z)_{-1}=0$ and it is a  dual mock Lie algebra {\rm \cite{Z17}}$\big).$
\end{lemma}

\begin{proof}
It is easy to see that
\begin{longtable}{lcl}
$\delta(xz)y+\delta x(yz)-(xy)z$&$=$&$0,$\\
$\delta(xy)z+\delta x(zy)-(xz)y$&$=$&$0.$
\end{longtable}\noindent 
Summarizing, we have
\ $(\delta-1)(xz)y+(\delta-1)(xy)z\ =\ 0,$ \ that gives our statement.
\end{proof}

\begin{lemma}\label{ancommass} 
Any anticommutative antiassociative algebra 
 is a  $\big(-\frac12\big)$-Lie algebra.
\end{lemma}

\begin{proof}
It is easy to see that in any anticommutative antiassociative algebra, we have
\begin{longtable}{lclclcl}
$(xz)y$&$=$&$-\ (zx)y$&$=$&$z(xy)$&$=$&$-\ (xy)z,$\\
$x(yz)$&$=$&$-\ (xy)z.$
\end{longtable}
\noindent Obviously,  we have $(xy)z \ =\   -\frac12 \big((xz)y+x(yz)\big).$
\end{proof}

\begin{lemma}\label{5} 
Let  $L$ be a     $\delta$-Lie  algebra $\big(\delta\neq-\frac12,1 \big).$ 
Then $L$ is a $2$-step nilpotent algebra.
\end{lemma}

\begin{proof}
    By Lemmas \ref{3} and \ref{ancommass}, $L$ is a   $\big(-\frac12\big)$-Lie algebra, 
    i.e. 
\begin{longtable}{lcr}
$(xy)z$&$=$&$ -\ \frac 12 \big((xz)y+x(yz) \big),$ \\ 
$(xy)z$&$=$&$ \delta\big((xz)y+x(yz)\big).$
\end{longtable}
\noindent Obviously, we have \ 
    $(1+2\delta)(xy)z\ =\ 0$  and $(xy)z\ =\ 0.$
\end{proof}

 \section{$\delta$-Lie   dialgebras}\label{ddilie}

By applying the standard KP $\big($Kolesnikov --- Pozhidaev$\big)$ algorithm for constructing di-identities for a fixed variety of algebras, we have di-identities for $\delta$-Lie dialgebras:
\begin{equation*}
x  \vdash y\ =\ -  y  \dashv x,\end{equation*}
\begin{equation*}
x  \dashv \big( y \vdash z \big)\ =\  x \dashv \big( y \dashv z \big), \quad  
\big(x  \vdash  y \big) \vdash z  \ =\  \big(x  \dashv  y \big) \vdash z,\end{equation*}
\begin{equation}\label{first}
x  \dashv ( y  \dashv z ) - \delta \big(  (x \dashv y ) \dashv z + y \vdash (x \dashv z) \big)\ =\ 0,
\end{equation}\begin{equation}
x  \vdash ( y  \dashv z ) - \delta \big(  (x \vdash y ) \dashv z + y \dashv (x \dashv z) \big)\ =\ 0,\\
\end{equation}\begin{equation}\label{third}
x  \vdash ( y  \vdash z ) - \delta \big(  (x \vdash y ) \vdash z + y \vdash (x \vdash z) \big)\ =\ 0.
\end{equation}

 \begin{theorem}\label{ddilieth}
     Let $L$ be a $\delta$-Lie dialgebra\footnote{The category of Lie dialgebras is equivalent to the category  of Leibniz algebras.} $\big(\delta\neq 1\big).$ 
     Then
     \begin{enumerate}
         \item if $\delta\neq -\frac 12,$ then $L$ is isomorphic to a $2$-step nilpotent algebra$;$ 
         \item if $\delta= -\frac 12,$ then $L$ is isomorphic to an 
         antiassociative anti-right-commutative algebra. 
     \end{enumerate}
 \end{theorem}

\begin{proof}
By using new one multiplication $xy:=x  \vdash y = -  y  \dashv x,$
 replacing 
$\{ x,y,z\} \rightarrow \{z,x,y\}$ in \eqref{first} and 
$x \leftrightarrow y$ in \eqref{third}, the last di-identities can be rewritten 
as follows

  \begin{longtable}{rcl}
$(yx)z $ & $=$ &$\delta \big( y(xz)+(yz)x \big),$\\
$\delta (yx)z$ & $=$ &$  (yz)x+ \delta y(xz),$\\
$\delta (yx)z$ & $=$ &$  y(xz)+ \delta (yz)x.$\\
\end{longtable}
Hence, we have the following situations:
\begin{enumerate}[(A)]
    
\item If $\delta=0$, then $(xy)z\ =\ x(yz) \ =\ 0.$ 
\item If $\delta\notin \big\{ 0, -\frac 12\big\},$ then summarizing the last two identities and 
subtracting $\frac{1+\delta}{\delta}$ times of the first identity, we have 
$\frac{2\delta^2-\delta-1}{\delta} (yx)z \ =\ 0.$
The last gives $(xy)z\ =\ x(yz)\ =\ 0.$ 

\item If $\delta=-\frac 12,$ then
$(xy)z\ = \ -\ x(yz) \ = \ -\ (xz)y.$
\end{enumerate}
\end{proof}

\subsection{Antiassociative anti-right-commutative algebras}

\begin{remark}\label{dimaac}
It is known that each antiassociative algebra is $3$-step nilpotent, and obviously, each $2$-step nilpotent algebra is antiassociative anti-right-commutative.
Taking the classification of $5$-dimensional antiassociative algebras from {\rm \cite{FKS}}, 
we can see that there are no non-$2$-step nilpotent antiassociative anti-right-commutative algebras with dimensional less than $6.$
Taking the classification of $6$-dimensional antiassociative anticommutative algebras from {\rm \cite{CKLS}}, we see that there are no non-$2$-step nilpotent  antiassociative anticommutative algebras with dimension less than $7.$
We can also present the following examples
of antiassociative anticommutative, and 
antiassociative $\big($non-anticommutative$\big)$ anti-right-commutative algebras in dimension $7:$ 

\begin{longtable}{|lcllllll|}
\hline
${\mathfrak A}$
&$:$& $e_1e_2=e_4$ & $e_1e_3=e_5$ & $e_1e_6=e_7$ & $e_2e_3=e_6$ & $e_2e_5=-e_7$ & $e_3e_4=e_7$ \\
& & $e_2e_1=-e_4$ & $e_3e_1=-e_5$ & $e_6e_1=-e_7$ & $e_3e_2=-e_6$ & $e_5e_2=e_7$ & $e_4e_3=-e_7$ \\
\hline${\mathfrak B}$ 
&$:$& $e_1e_1=e_7$ &\multicolumn{5}{l|}{and the multiplication rules from ${\mathfrak A}$}  \\
\hline
\end{longtable}   
\noindent But the question of the existence of a $6$-dimensional non-$2$-step nilpotent antiassociative  $\big($non-anticommutative$\big)$ anti-right-commutative  algebra is still open.
If it exists, it should be a one-dimensional central extension of a $2$-step nilpotent $5$-dimensional algebra, but its classification has not obtained.  
\end{remark}

\begin{proposition}\label{basaac}
        Let $\mathfrak{aar}(X)$ be the free antiassociative anti-right-commutative algebra, generated by the set $X=\big\{x_i \big\}_{i \in I}.$ 
Then the following set of elements 
\begin{center}
    ${\mathfrak B}=\big\{  \ x_i,\  x_{i}x_{j},\  x_{i} (x_{j_1}x_{j_2})  \ \big\}_{i, j, j_1,j_2 \in I}^{ j_1 < j_2}$
\end{center}
gives a basis of $\mathfrak{aar}(X).$ 
    \end{proposition}
\begin{proof}
Let us remember that $\mathfrak{aar}(X)$ is a $3$-step nilpotent algebra, 
hence, $x_{k_1}\ldots x_{k_t}=0$ for $t \ge 4.$
All basis elements are constructed from words of length $1$, $2$, or $3.$
Due to antiassociativity and anti-right-commutativity, all words given by 
$\big\{\ x_{i} (x_{j_1}x_{j_2})\  \big\}_{i,   j_1,j_2 \in I}^{ j_1 < j_2}$ are independent.
Hence, ${\mathfrak B}$ is a basis of  $\mathfrak{aar}(X).$ \end{proof}

Let us construct multilinear base elements of degree $3$ for a free antiassociative anti-right-commutative algebra. 
Below, we give a presentation of $9$ non-base elements of degree $3$ as a linear
combination of the base elements of degree $3$:

\begin{longtable}{rclrclrcl}
$a (c  b) $&$ =$&$-a (bc),$ & $(a b)  c $&$ =$&$-a (bc),$ &  $(a c)  b $&$ =$&$ a (bc),$ \\
$b (c  a) $&$ =$&$-b (ac),$ & $(b a)  c $&$ =$&$-b (ac),$ &  $(b c)  a $&$ =$&$ b (ac),$ \\
$c (b  a) $&$ =$&$-c (ab),$ & $(c a)  b $&$ =$&$-c (ab),$ &  $(c b)  a $&$ =$&$ c (ab).$ \\
\end{longtable}

Following the approach in \cite{GK94}, we compute the dual operad $\mathfrak{aar}^!$, where the operad $\mathfrak{aar}$ governs the variety of antiassociative anti-right-commutative  algebras. Then,
 
\begin{longtable}{rcl}
$\big[[a \otimes x, b \otimes y], c \otimes z\big] $&$+$&$\big[[b \otimes y, c \otimes z], a \otimes x\big] \ +\  \big[[c \otimes z, a \otimes x], b \otimes y\big]$ \\
&$=$&  $((a  b)  c) \otimes ((x \bullet y) \bullet z) - (c  (a  b)) \otimes (z \bullet (x \bullet y))$ \\
&&$\quad - ((b  a)  c) \otimes ((y \bullet x) \bullet z) + (c  (b  a)) \otimes (z \bullet (y \bullet x))$ \\
&&$\quad + ((b  c)  a) \otimes ((y \bullet z) \bullet x) - (a  (b  c)) \otimes (x \bullet (y \bullet z))$ \\
&&$\quad - ((c  b)  a) \otimes ((z \bullet y) \bullet x) + (a  (c  b)) \otimes (x \bullet (z \bullet y))$ \\
&&$\quad + ((c  a)  b) \otimes ((z \bullet x) \bullet y) - (b  (c  a)) \otimes (y \bullet (z \bullet x))$ \\
&&$\quad - ((a  c)  b) \otimes ((x \bullet z) \bullet y) + (b  (a  c)) \otimes (y \bullet (x \bullet z))\ =$\\

\multicolumn{3}{l}{$\mbox{
\big(
by using the above presentation of the $9$ non-basis elements of degree $3$}$}\\ 
\multicolumn{3}{r}{$\mbox{
as a linear combination of basis elements, we have \big)}$}\\

&$=$& 
$(b(ac)) \otimes \big((y \bullet x) \bullet z +(y \bullet z) \bullet x + y \bullet (z \bullet x)+y \bullet (x \bullet z)  \big) -$\\

& & 
$(a(bc)) \otimes \big((x \bullet y) \bullet z  +x \bullet (y \bullet z) + x \bullet (z \bullet y) + (x \bullet z) \bullet y  \big) - $\\

&& 
$(c  (a  b)) \otimes  \big( z \bullet (x \bullet y)+  z \bullet (y \bullet x) + (z \bullet y) \bullet x + (z \bullet x) \bullet y \big).$

\end{longtable}
Hence, the dual  operad $\mathfrak{aar}^!$ is  equivalent to
 an operad that governs the anti-right-alternative identity\footnote{Also known as the pre-Jacobi-Jordan identity \cite{BBMM}.}:
\begin{longtable}{rcl}
$\big(x , y,  z\big)_{-1}^\bullet  $&$=$&$- \ \big(x , z,  y\big)_{-1}^\bullet.$
\end{longtable}

\begin{proposition}
Dual antiassociative anti-right-commutative algebras are Jacobi-Jordan admissible.    
\end{proposition} 

\begin{proof}
    Thanks to \cite{BBMM}, each Jacobi-Jordan admissible algebra satisfies 
  \begin{center}
        $\sum\limits_{\sigma \in \mathbb S_3} \big(x_{\sigma(1)}, x_{\sigma(2)}, x_{\sigma(3)} \big)_{-1}=0.$
  \end{center}
 In particular, each algebra that satisfies 
    $\big(x , y,  z\big)_{-1}  + \big(x , z,  y\big)_{-1} =0$ is Jacobi-Jordan admissible.
    Hence, each dual antiassociative anti-right-commutative algebra is Jacobi-Jordan admissible.    \end{proof}

 \section{$\delta$-Leibniz algebras}\label{dl}

 \subsection{Anti-Leibniz algebras}\label{antiL}

\begin{definition}[see \cite{antileib}]\label{deltantiLeibniz}
Let $L$ be an algebra, then   $L$ is called a (right)  anti-Leibniz  algebra if the following identity holds true:
    \begin{equation*} \label{antiLeibIdentity}
       (xy)z\ =\  -  (xz)y - x(yz).
\end{equation*}
\end{definition}

\begin{theorem}\label{simplantil}
    There are no simple anti-Leibniz algebras.
\end{theorem}

\begin{proof}
    Firstly, we observe that the vector space 
    $[L,L]_{1}$ gives an ideal of $L$.
    It follows from some direct computations given below:

\begin{longtable}{lclclcl}
    
$x[y,z]_1$&$=$&$x(yz)-x(zy)$ &  $=$ &
$ - (xy)z-(xz)y+(xz)y+(xy)z$ & $=$ & $0;$\\ 

$[y,z]_1x$&$=$&$(yz)x-(zy)x$ & $=$&
$-\big((yx)z+y(zx)-(zx)y-z(yx)\big)$ & $=$&
$[zx,y]_1+[z,yx]_1.$ 
\end{longtable}

Hence, we have two opportunities: 
\begin{enumerate}
    \item[{\rm (A)}]  $[L,L]_{1} =\big\{0\big\},$
    i.e., $L$ is a commutative algebra. 
    But each commutative anti-Leibniz algebra is Jacobi-Jordan, 
    i.e., a Jordan nil-algebra, and it is non-simple.
 \item[{\rm (B)}]  $[L,L]_{1} =L,$ but it means that 
\begin{center}
    $xy\ =\ \big(\sum\limits_{i=1}^n[x_i,y_i]_1\big) 
\big( \sum\limits_{j=1}^m[z_j,t_j]_1\big) \ = 
\  \sum\limits_{1\le i\le n, 1\leq j \leq m}[x_i,y_i]_1[z_j,t_j]_1 \ = \ 0.$
\end{center}
It gives that $L$ has the trivial multiplication and can not be simple.\end{enumerate}\end{proof}

\begin{theorem}\label{cenrext}
    Each symmetric anti-Leibniz algebra\footnote{i.e., left and right anti-Leibniz algebra.} is a central extension of a suitable Jacobi-Jordan algebra. 
\end{theorem}

\begin{proof}
Let $L$ be a symmetric anti-Leibniz algebra. 
Doing some computation similar to the proof of the previous theorem,  it is easy to see that 
\begin{center}
    $x[y,z]_1\ =\ 0$ and $[y,z]_1x \ = \ 0.$
\end{center}
Hence, $ \big\langle [y,z]_1 \ | \  y,z \in L \big\rangle$ is in the annihilator of $L.$
It follows that $L$ is a central extension of $L \big/{\rm Ann}(L)$
and $L \big/{\rm Ann}(L)$ is commutative, hence $L \big/{\rm Ann}(L)$  is Jacobi-Jordan.
    \end{proof}

 \subsection{$\delta$-Leibniz algebras}\label{deltL}

\begin{definition}\label{deltaLeibniz}
Let $L$ be an algebra and $\delta$ be a fixed complex number. Then   $L$ is called a  $\delta$-Leibniz  algebra if the following identity holds true:
    \begin{equation} \label{antiLeibIdentity}
       (xy)z\ =\  \delta \big((xz)y+x(yz)\big).
\end{equation}
\end{definition}

Let $\mathcal{L}eib_{\delta}$ be a variety of $\delta$-Leibniz algebras  for some $\delta$. Then $\mathcal{L}eib_0$ is the variety of  $2$-step right nilpotent algebras, and $\mathcal{L}eib_1$ is the variety of Leibniz  algebras.
Hence, the varieties of $\delta$-Leibniz   algebras give a type of continuous deformation 
of the variety of Leibniz algebras to the variety of  $2$-step right nilpotent algebras and vice versa.

\begin{theorem}
Let $L$ be a $\delta_1$- and $\delta_2$-Leibniz algebra, where $\delta_1 \neq \delta_2$. Then $L$ is a $2$-step nilpotent algebra.
\end{theorem}

\begin{proof}
It is easy to see that
$$(xy)z\ =\ \delta_1 \big((xz)y+x(yz)\big) \ =\ \delta_2 \big((xz)y+x(yz)\big).$$
Since $\delta_1$ and $\delta_1$ are different, we obtain the following identity:
\begin{center}
    $(xz)y+ x(yz)\ =\ 0.$
    \end{center}

From this, we obtain the relation $(xy)z=0$, and from the definition of  $\delta$-Leibniz algebras, we derive the identity $x(yz)=0$. Therefore, $L$ is a 2-step nilpotent algebra. \end{proof}

\begin{corollary}    
 $\bigcap\limits_{\delta} \mathcal{L}eib_{\delta}$ is the variety of $2$-step nilpotent algebras.
\end{corollary}

The classification of two-dimensional algebras over an algebraically closed field was established in \cite{kv16}. We present a classification of two-dimensional $\delta$-Leibniz algebras\footnote{In particular, it corrects a previously incorrect classification of anti-Leibniz algebras given in \cite[Theorem 2.3]{antileib}.}, for which direct computations can be achieved.

\begin{theorem}\label{cldn}
    Let $L$ be a nontrivial $2$-dimensional $\delta$-Leibniz algebra, then it is isomorphic to one of the algebras listed below

\begin{longtable}{|l|c|l|l|l|l|}

\hline
 ${\bf A}_3$ &$\delta$ & $e_1e_1= e_2$&  
 $e_1e_2= 0$ & $e_2e_1=0 $& $e_2e_2=0$ \\

   \hline
${\bf B}_2(1)$ &$0$ &  $e_1e_1= 0$&  
 $e_1e_2= 0$ & $e_2e_1=e_1 $& $e_2e_2=0$ \\

\hline
 ${\bf B}_2(0)$ &$1$ &  $e_1e_1= 0$&  
 $e_1e_2= e_1$ & $e_2e_1=0 $& $e_2e_2=0$ \\
  
 \hline
 ${\bf B}_3$ &$1$ &  $e_1e_1= 0$&  
 $e_1e_2= e_2$ & $e_2e_1=-e_2 $& $e_2e_2=0$ \\

  \hline
 ${\bf D}_2(0,0)$ &$\frac 12$ &  
$e_1e_1= e_1$&   $e_1e_2= 0$ & 
$e_2e_1=0 $& $e_2e_2=0$ \\

  \hline
 ${\bf D}_2(0,1)$ &$\frac 12$ &  
$e_1e_1= e_1$&   $e_1e_2= 0$ & 
$e_2e_1=e_2 $& $e_2e_2=0$ \\

\hline
 ${\bf D}_2(1,1)$ &$\frac 12$ &  
$e_1e_1= e_1$&   $e_1e_2= e_2$ & 
$e_2e_1=e_2 $& $e_2e_2=0$ \\

\hline
${\bf E}_1(1,0,0,1)$ &$\frac 12$ &  
$e_1e_1= e_1$&   $e_1e_2= e_1$ & 
$e_2e_1=e_2 $& $e_2e_2=e_2$ \\

\hline
 ${\bf E}_1(0,0,0,0)$ &$\frac 12$ &  
$e_1e_1= e_1$&   $e_1e_2= 0$ & 
$e_2e_1=0 $& $e_2e_2=e_2$ \\


\hline

\end{longtable}
\noindent Here: 
notations in  the first column are  given from {\rm \cite{kv16}} and 
the number $\gamma$ in the second column means that it is a $\gamma$-Leibniz algebra. 

\end{theorem}

\begin{proposition}\label{nilpL}
    Let $L$ be an $n$-dimensional nilpotent $\delta$-Leibniz algebra $\big( \delta \notin\{ 0,1\} \big)$,
    then the nilpotency index of $L$ is not greater than $n.$
\end{proposition}

\begin{proof}
 It is easy to see that each product of $k$ elements in a $\delta$-Leibniz algebra can be written as
 a sum of products with left-ordered brackets.
 Hence,  
 \begin{center}
     $ \sum\limits_{i+j=k} L^iL^j \ := \ 
     L^k \subseteq \underbrace{\big(\ldots ((LL)L)\ldots\big)L}_{k \  times}\ =:\ L^{[k}.$ 
\end{center}
Obviously, 
$ 0 = L^{[n+1} \subseteq L^{[n-1} \subseteq \ldots \subseteq L^{[2} \subseteq L^{[1}=L.$
Hence, the nilpotency index of $L$ is not greater than $n+1$. 
 Thanks to \cite[Theorem 2.1]{MO14}, there are no $n$-dimensional $\delta$-Leibniz algebras with the nilpotency index $n+1,$
 hence the nilpotency index of $L$ is not greater than $n.$
\end{proof}

\begin{definition}[see \cite{GK25}]\label{lengthdef}
Let ${\rm A}$ be a finite-dimensional algebra generated by a set $S.$
We denote $L(S) = \bigcup\limits_{i=1}^\infty \big\langle S^i \big\rangle.$ Since the set $S$ is generating for ${\rm A}$, we have ${\rm A} = L(S).$ 
The length of a generating set $S$ of  ${\rm A}$ is defined as follows: 
$l(S) = {\rm min} \big\{k \in \mathbb N : \bigcup\limits_{i=1}^k \big\langle S^i \big\rangle = {\rm A}\big\}.$ 
The length of an algebra ${\rm A}$ is 
$l({\rm A}) = {\rm max}\big\{l(S) : L(S) = {\rm A}\big\}.$
\end{definition}

\begin{corollary}\label{lengthL}
    The length of an $n$-dimensional  $\delta$-Leibniz algebra $\big(\delta \neq 0,1\big)$
  is not greater than $n-1.$ 
\end{corollary}

\begin{proof}
    Let $L$ be an $n$-dimensional $\delta$-Leibniz algebra $\big(\delta \neq 0,1\big)$, such that $l({L})=k.$
    Obviously\footnote{The idea is completely similar to the proof of \cite[Lemma 38]{KKL23}.}, there exists an $n$-dimensional nilpotent $\delta$-Leibniz algebra $L_0$ $\big(\delta \neq 0,1\big)$, such that $l({L_0})=k.$
    Thanks to Proposition  \ref{nilpL}, $k\le n-1$, and the statement is proved. 
\end{proof}

\begin{proposition}\label{annid}
Let $L$ be a $\delta$-Leibniz algebra $\big(\delta\neq0\big);$  $I$ and $J$ be  ideals of $L$, then
\begin{enumerate}
    \item[\rm{(A)}] 
${\rm Ann}_r(L)\ =\ \big \{x\in L \ | \ Lx=0\big\}$ is     a two-sided   ideal of $L;$

\item[\rm{(B)}] 
$I  J+JI \ = \ \big\{ x_1y_1+y_2x_2 \ |\  x_1,x_2\in I, y_1,y_2\in J\big\}$  is   a two-sided   ideal of $L$.

\end{enumerate}

\end{proposition}
\begin{proof} 
Firstly, for any $a\in {\rm Ann}_r(L)$ and $x\in L$ we have 
\begin{center}$ y(ax)\ = \ \delta^{-1}(ya)x- (yx)a\ =\ 0$ and $xa=0.$  \end{center} 
Hence, $xa, \ ax\in {\rm Ann}_r(L)$. Thus, ${\rm Ann}_r(L)$ is an ideal of $L.$

Secondly, for any $i_1,i_2\in I, j_1,j_2 \in J$ and $x \in L$ we have 
\begin{longtable}{lccc}
$(i_1j_1+j_2i_2)x$&$ =$&$ (i_1x)j_1+i_1(j_1x)+(j_2x)i_2+j_2(i_2x)$ &$\in I J+JI;$\\  
$x(i_1j_1+j_2i_2)$&$ =$&$ \delta^{-1} (xi_1)j_1-(xj_1)i_1+\delta^{-1} (xj_2)i_2-(xi_2)j_2$&$\in I J+JI.$\\  
\end{longtable}
\noindent Hence,   $I J+JI$ is an ideal of $L.$
\end{proof}  

\begin{remark}
The restriction on  $\delta\neq0$ in Proposition \ref{annid} is important. 
Let us consider a $2$-dimensional algebra ${\mathfrak L}$ given by the multiplication table:
$e_1e_2=e_2.$
This is a $0$-Leibniz algebra such that 
${\mathfrak L}\big({\rm Ann}_r({\mathfrak L}) {\mathfrak L}\big) \neq 0.$
\end{remark}

\begin{proposition}\label{deltaass}
    Let ${\rm A}$ be a $\delta$-associative algebra $\big( \delta \neq 1\big),$ then  ${\rm A}$ is $3$-step nilpotent.
\end{proposition}

\begin{proof}
 It is easy to see that 
\begin{center}
     $ \big((xy)z\big)t \ = \ 
 \delta \big(x(yz)\big)t \ = \ 
 \delta^2 x \big((yz)t\big) \ = \ 
 \delta^3 x \big(y(zt)\big) \ = \ 
 \delta^2 (xy)(zt) \ = \ 
 \delta \big((xy)z\big)t,$
\end{center}
which gives that $\big((xy)z\big)t=0$ and ${\rm A}$ is $3$-step nilpotent.\end{proof}

\begin{theorem}\label{comL} 
The variety of  commutative $\delta$-Leibniz algebras $\big(\delta\neq-1,1 \big)$
coincides with the variety of 
commutative $\frac{\delta}{1-\delta}$-associative algebras.
In particular, the variety of commutative $\frac 12$-Leibniz algebras coincides with the variety of commutative associative   algebras;
the variety of commutative $\delta$-Leibniz algebras  
$\big(\delta \neq \frac 12 \big)$ 
is a subvariety in the variety of commutative $3$-step nilpotent algebras.
\end{theorem}

\begin{proof}
It is easy to see that each commutative $\delta$-Leibniz algebra satisfies
\begin{longtable}{lcl}
$\delta(xz)y+\delta x(yz)-(xy)z$&$=$&$0,$\\ 
$\delta(xy)z+\delta x(zy)-(xz)y$&$=$&$0.$
\end{longtable}
Summarizing,  the equation obtained by multiplying  the first equation by $\delta$ and the second equation, we have the following identity:
\begin{center}
    $(\delta^2-1)(xz)y+(\delta^2+\delta)x(zy)\ =\ 0,$
    \end{center}
that gives 
$(xz)y- \frac{\delta}{1-\delta} x(zy)=0,$
i.e., 
our algebra satisfies the $\frac{\delta}{1-\delta}$-associative identity.

On the other hand, each commutative  $\frac{\delta}{1-\delta}$-associative algebra satisfies 

\begin{center}
$(xy)z \ = \ 
\frac{\delta}{1-\delta} x(yz) \ = \ 
\frac{\delta^2}{1-\delta} x(yz)+ \delta  x(yz) \ = \ 
\delta   \big( (xz)y + x(yz) \big),$
\end{center}
i.e., it satisfies the $\delta$-Leibniz identity. \end{proof}

\begin{theorem}
    Let $L$ be a $\delta$-Leibniz algebra $\big(\delta\neq-1 \big).$ 
    Then $L$ is a $\frac{\delta}{1+\delta}$-right-symmetric algebra\footnote{i.e., it satisfies 
    $\big(x,y,z\big)_{\frac{\delta}{1+\delta}}\ =\ \big(x,z,y\big)_{\frac{\delta}{1+\delta}},$ \ see \cite{dN}.}. 
\end{theorem}

\begin{proof}
It is easy to see that
\begin{longtable}{lcl}
$(xy)z$&$=$&$\delta \big((xz)y+x(yz)\big),$\\ 
$(xz)y$&$=$&$\delta \big((xy)z+x(zy)\big).$\\ 
\end{longtable}
Subtracting these two identities, we obtain the following identity:
\begin{center}
    $(1+\delta)(xy)z-\delta x(yz)\ =\ (1+\delta)(xz)y-\delta x(zy).$
    \end{center}
Dividing this equality by $1+\delta$, we obtain the following identity:
\begin{center}
    $(xy)z-\frac{\delta}{1+\delta}x(yz)\ =\ (xz)y-\frac{\delta}{1+\delta} x(zy).$
        \end{center}
Hence, each $\delta$-Leibniz algebra is a $\frac{\delta}{1+\delta}$-right-symmetric algebra. \end{proof}

\begin{proposition}\label{liadmd}
    Let $L$ be a $\delta$-Leibniz algebra, 
    then $L$ is Lie admissible if and only if it satisfies $\sum\limits_{\sigma \in \mathbb S_3}  (-1)^{\sigma} \big(x_{\sigma(1)} x_{\sigma(2)}\big) x_{\sigma(3)}\ =\ 0.$
\end{proposition}

\begin{proof}
    It is easy to see that $L$ satisfies

    \begin{center}
        $\big(x,y,z\big) \ =\  (1+\delta^{-1}) (xy)z+(xz)y.$
    \end{center}
Hence $\big(x,y,z\big) - \big(x,z,y\big) = \delta^{-1} \big((xy)z-(xz)y \big).$
It follows, 
\begin{center}
$\sum\limits_{\sigma \in \mathbb S_3} (-1)^{\sigma}\big (x_{\sigma(1)},x_{\sigma(2)}, x_{\sigma(3)} \big)\ =\  
\delta^{-1}\sum\limits_{\sigma \in \mathbb S_3}  (-1)^{\sigma} \big(x_{\sigma(1)} x_{\sigma(2)}\big) x_{\sigma(3)},$
\end{center}
 where the left side gives the Lie-admissible identity.
Hence, our statement is proved. \end{proof}

In 1972, Kantor introduced $\big($left$\big)$ conservative algebras as a generalization of Jordan algebras \cite{K72}
 (also, see surveys on recent studies of conservative algebras and superalgebras \cite{P20,k23}). Namely, a vector space ${\rm V}$ with a multiplication $\cdot$ is called a $\big($left$\big)$ {\it conservative algebra} if there is a new multiplication $*:{\rm V}\times {\rm V}\rightarrow {\rm V}$ such that
\begin{center}
$\big[L_b^\cdot, [L_a^\cdot,\cdot]\big]\ =\ -\big[L_{a*b}^\cdot,\cdot\big],$
\ where \ 
$\big[L_c^\cdot, {\rm F}\big] (x,y)\ =\ c\cdot{\rm F}(x, y)-{\rm F}(c\cdot x,  y)-{\rm F}(x, c\cdot y).$
\end{center}
In other words, the following identity holds for all $a,b,x,y\in {\rm V}$:
\begin{multline*}\label{tojdestvo_glavnoe}
b\big(a(xy)-(ax)y-x(ay)\big)-
a\big((bx)y\big)+\big(a(bx)\big)y+(bx)(ay)\\
-a\big(x(by)\big)+(ax)(by)+x\big(a(by)\big)
+(a *b)(xy)-\big((a* b)x\big)y-x\big((a *b)y\big) \ =\ 0.
\end{multline*}

\begin{theorem}\label{Lcons}
    Each $\big($left$\big)$ $\delta$-Leibniz algebra is a $\big($left$\big)$ 
    conservative algebra  with the same additional multiplication.
\end{theorem}

\begin{proof}
If $\delta=0,$ then the algebra satisfies $a(xy)=0$ and obviously it is left conservative.
 Each left $\delta$-Leibniz algebra satisfies the identity $(ax)y=\delta^{-1} a(xy)-x(ay).$ Hence,

\begin{longtable}{rcl}
$b\big(a(xy)-(ax)y-x(ay)\big)$ &$=$ &$b\big(a(xy)\big)-\delta^{-1}b\big(a(xy)\big)
,$\\
$\big(a(bx)\big)y+(bx)(ay)-a\big((bx)y\big)$&$=$&$\delta^{-2} a\big(b(xy)\big)-\delta^{-1}a\big(x(by)\big)
-\delta^{-1}a\big(b(xy)\big)+a\big(x(by)\big), $\\
$x\big(a(by)\big)-a\big(x(by)\big)+(ax)(by)$&$=$&$
-a\big(x(by)\big)+\delta^{-1}a\big(x(by)\big)
, $\\
$(ab)(xy)-\big((a b)x\big)y-x\big((a b)y\big)$&$=$&$ (1-\delta^{-1})(ab)(xy),$
\end{longtable}
\noindent and summarizing the left and right parts, we have 
the conservative identity from the left side and $0$ from the right side.
Hence,  each $\big($left$\big)$ $\delta$-Leibniz algebra is a $\big($left$\big)$ conservative algebra with the same additional multiplication.\end{proof}

\subsection{$\delta$-Leibniz algebras via $\delta$-associative algebras and dialgebras}\label{ddd}

It is known that each associative  algebra under the usual commutator product gives a Lie algebra $\big($in particular, it gives a Leibniz algebra$\big),$ 
and each antiassociative algebra under the symmetric product gives a mock-Lie algebra 
$\big($in particular, it gives an anti-Leibniz algebra$\big).$ 
Our next statement gives a generalization of above mention result.

\begin{proposition}
    Let ${\rm A}$ be a $\delta$-associative algebra $\big(\delta \neq 0\big)$, then
    $\big({\rm A}, \  [\cdot,\cdot]_{\delta}\big)$ is
    a non-$2$-step nilpotent $\delta$-Lebniz algebra if and only if 
    $\delta^2=1.$
\end{proposition}

\begin{proof}
The relations given below follow from the definitions of $\delta$-commutator and $\delta$-associative algebras. 
\begin{longtable}{rcrclclcl}
$\big[[x,y]_{\delta},\  z\big]_{\delta}$ &$=$&
$[xy- \delta yx, \ z]_\delta$ &$=
$&$(xy)z-\delta (yx)z -\delta z (xy)+\delta^2 z(yx)$& $=$& \\
&&&$=$&$\delta x(yz)-\delta^2 y(xz) -\delta z (xy)+\delta^2 z(yx);$\\

$-\delta \big[[x,z]_{\delta},\  y\big]_{\delta}$ &$=$&
$-\delta [xz- \delta zx, \ y]_\delta$ &$=
$&$-\delta  \big((xz)y-\delta (zx)y -\delta y (xz)+\delta^2 y(zx) \big)$& $=$& \\
&&&$=$&$-\delta^2 x(zy)+\delta^3 z(xy) +\delta^2 y (xz)-\delta^3 y(zx);$\\

$-\delta \big[x,\ [y,z]_\delta\big]_\delta$ & $=$&
$-\delta [x,\  yz - \delta zy]_\delta$ & $=$&
$-\delta \big(x (yz) - \delta x (zy) - \delta (yz)x +\delta^2 (zy)x \big)$& $=$& \\
&&&$=$&$-\delta x (yz) + \delta^2 x (zy) + \delta^3 y(zx) -\delta^4 z(yx)$. \\
\end{longtable}

Summarizing, we have

\begin{center}
$\big[[x,y]_{\delta},\  z\big]_{\delta}
-\delta \big[[x,z]_{\delta},\  y\big]_{\delta}-\delta \big[x,\ [y,z]_\delta\big]_\delta\ = \ 
-\delta (1-\delta^2) z [x,y]_\delta.$
    \end{center}
It is easy to see that if an algebra satisfies $ z [x,y]_\delta=0,$  then 
\begin{center}
    $z(yx) \ = \ \delta z (xy) \ = \ \delta^2 z(yx).$
\end{center} 
If  $\delta^2\neq 1,$ it has to be $2$-step nilpotent,  
    which proves our statement.
\end{proof}

Dzhumadildaev proved that each Leibniz algebra under the symmetric product gives a metabelian commutative algebra \cite{dzh08}.
Our next statement gives a generalization of above mention result.
In particular, we prove that each anti-Leibniz algebra under the usual commutator multiplication gives a metabelian anticommutative algebra.

\begin{proposition}
    Let $L$ be a $\delta$-Leibniz algebra $\big(\delta \neq 0\big)$, then
    $\big(L, \  [\cdot,\cdot]_{-\delta}\big)$ is metabelian if and only if 
    $\delta^2=1$ or $\delta^2\neq 1$ and $L$ satisfies
    \begin{center}
        $\big([x,y]_{-\delta}z\big)t+ \delta \big([z,t]_{-\delta}x \big)y \ = \ 0.$
    \end{center}
\end{proposition}

\begin{proof}
The relations given below follow from the definitions of $\delta$-commutator and $\delta$-Leibniz algebras. 
\begin{longtable}{lcl}
$\big[[x,y]_{-\delta},\ [z,t]_{-\delta}\big]_{-\delta}$ & $=$&
$[xy+\delta yx,\  zt +\delta tz]_{-\delta} \ = $\\
&$=$&$ (xy)(zt)+\delta (xy)(tz)+\delta (yx)(zt)+\delta^2 (yx)(tz)+$\\
\multicolumn{3}{r}{$+ \delta (zt)(xy)+ \delta^2(tz)(xy)+\delta^2 (zt)(yx)+\delta^3 (tz)(yx)\ =$}\\

&$=$&$ \delta^{-1}((xy)z)t - ((xy)t)z +\delta ((xy)t)z-\delta ((xy)z)t+$\\
&&\multicolumn{1}{c}{$((yx)z)t-\delta((yx)t)z+\delta ((yx)t)z-\delta^2 ((yx)z)t+$}\\
&&\multicolumn{1}{c}{$((zt)x)y-\delta ((zt)y)x)+\delta((zt)y)x-\delta^2((zt)x)y+$}\\
\multicolumn{3}{r}{$\delta ((tz)x)y-\delta^2((tz)y)x+\delta^2((tz)0y)x-\delta^3 ((tz)x)y \ =$}\\

&$=$&$ \big(\delta^{-1}-\delta\big)\big( ((xy)z)t+\delta ((yx)z)t +\delta ( ((zt)x)y+\delta ((tz)x)y) \big).$\\
\end{longtable}
\noindent The analysis of equivalent relations gives our statement.\end{proof}



It is known that an associative (resp., antiassociative) dialgebra under new multiplication
$xy= x \dashv y - y \vdash x$ (resp., $xy= x \dashv y + y \vdash x$) gives a Leibniz (resp., anti-Leibniz) algebra \cite{antileib}.
The following statement generalizes the above-mentioned results.

\begin{definition}
    Let $L$  be a dialgebra and $\delta$ be a fixed complex number.
    Then $L$ is  a $\delta$-associative dialgebra if the following identities hold true:

    \begin{longtable}{lclclcl}

&&$x  \dashv ( y \vdash z )-  x \dashv ( y \dashv z ) $& $=$&$0,$\\
&&$(x  \vdash  y ) \vdash z - (x  \dashv  y ) \vdash z$& $=$&$0,$\\
$(x,y,z)^\dashv_\delta$ &$:=$ &$ ( x \dashv y ) \dashv z - \delta x \dashv (y \dashv z)$& $=$&$0,$\\
$(x,y,z)^\vdash_\delta$ &$:=$ &$ ( x \vdash y ) \vdash z - \delta x \vdash (y \vdash z)$& $=$&$0,$\\
$(x,y,z)^\times_\delta$ &$:=$ &$ ( x \vdash y ) \dashv z - \delta x \vdash (y \dashv z)$& $=$&$0.$\\
    \end{longtable}
\end{definition}

\begin{theorem}
    Let $L$ be a $\delta$-associative dialgebra $\big(\delta\neq0\big),$ 
    then $L$ with a new multiplication $\llbracket x,y \rrbracket_{\delta}= x \dashv y -\delta  y \vdash x$ is a $\delta$-Leibniz algebra if and only if 
    it satisfies 
    \begin{center}$\delta (x,y,z)_{\delta^{-1}}^\vdash = (x,z,y)_{\delta^{-1}}^\times.$
\end{center}
\end{theorem}

\begin{proof}
    It is easy to see that 
    \begin{longtable}{lcl}
    $ \big\llbracket  \llbracket x,y \rrbracket_{\delta}, z \big\rrbracket_{\delta} $&$-$&$\delta 
    \big( \big\llbracket  \llbracket x,z \rrbracket_{\delta}, y \big\rrbracket_{\delta} + \big\llbracket x, \llbracket y,z \rrbracket_{\delta} \big\rrbracket_{\delta} \big)$\\
    &$=$&$ (x \dashv y)  \dashv z - \delta z  \vdash ( x  \dashv y) - \delta (y  \vdash x)  \dashv z +\delta^2 z  \vdash ( y  \vdash x)  $\\
    && $-\delta (x  \dashv z )  \dashv  y +\delta^2 y  \vdash (x  \dashv z)+\delta^2 (z  \vdash x)  \dashv y - \delta^3 y  \vdash (z  \vdash x) $\\
    && $-\delta x  \dashv (y  \dashv z) +\delta^2 (y  \dashv z)  \vdash x +\delta^2 x  \dashv (z  \vdash y) - \delta^3 (z  \vdash y )  \vdash x$\\
  &$=$&$(x,y,z)^\dashv_{\delta} + \delta^2 (z,x,y)^\times_{\delta^{-1}} - \delta(y,x,z)^\times_{\delta}- \delta^3 (z,y,x)^\vdash_{\delta^{-1}}-\delta (x,z,y)^\dashv_{\delta}+\delta^2(y,z,x)^\vdash_\delta$\\
  &$=$&$ \delta^2 (z,x,y)^\times_{\delta^{-1}}  - \delta^3 (z,y,x)^\vdash_{\delta^{-1}} $\end{longtable}\end{proof}

 \subsection{Symmetric $\delta$-Leibniz algebras}\label{syml}

\begin{theorem}\label{symmetric}
     Let $L$ be a symmetric $\delta$-Leibniz algebra\footnote{i.e., left and right $\delta$-Leibniz algebra.}. 
     Then
     \begin{enumerate}
         \item[$({\rm A})$] if $\delta\neq \frac 12,$ then $x^n=0$ for all $n>2$. 
         \item[$({\rm B})$] if $\delta= \frac 12,$ then $L$ is power associative. 
     \end{enumerate}
 \end{theorem}

 \begin{proof}
It is easy to see that
$x^2x=\delta \big(x^2x+xx^2\big)=xx^2.$
From this, we have  $x^2x=2\delta x^2x.$ Then
     \begin{enumerate}
         \item[$({\rm A})$] if $\delta\neq \frac 12,$ then $x^3=x^2x=xx^2=0$  and $x^2x^2= (xx)x^2=\delta(x^3x+xx^3)=0,$ which gives $x^n=0$ for all $n>2.$ 
         \item[$({\rm B})$] if $\delta= \frac 12,$ then we have $x^3=x^2x=xx^2.$ 
         It is easy to see that 
         \begin{center}
             $x^2x^2  \ = \ \frac12 \big (x^3x+xx^3\big) \ = \ \frac 12 \big(x^3x+\frac 12 x^2x^2 +\frac 12 x x^3\big),$ \  i.e. $x^2x^2=x^3x.$
         \end{center}         
    Then, the   algebra $L$ satisfies the conditions  from \cite[Lemma 3]{Albert} and hence, $L$ is power-associative.  
     \end{enumerate}
\end{proof}

\begin{corollary} Let $L$ be a nonzero complex $n$-dimensional one-generated symmetric $\delta$-Leibniz algebra. Then 

\begin{enumerate}
    \item[$({\rm A})$] if $\delta \neq \frac 12,$ then $n=2$ and it has the multiplication table given by $e_1^2=e_2.$\

\item[$({\rm B})$] if $\delta=\frac12,$ then it has the multiplication table given by  $e_ie_j=e_{i+j}, \ \ 2\leq i+j\leq n.$ 
\end{enumerate}
\end{corollary}

 \begin{proposition}
      Let $L$ be a symmetric $\delta$-Leibniz algebra, 
     then $\big(L, [\cdot, \cdot]_1\big)$ is a $\delta$-Lie algebra.
In particular, if $\delta\neq-\frac12,$ then $\big(L, [\cdot, \cdot]_1\big)$ is a $2$-step nilpotent Lie algebra.
 \end{proposition}

 \begin{proof}
 It is easy to see that
 \begin{longtable}
     {llllllllllll}
     $\big[[x,y]_1,z \big]_1$ &$=$&$  (xy-yx)z-x(xy-yx)\ =$ \\
 &$=$&$\delta\big((xz)y+x(yz)-(yz)x-y(xz)-(zx)y-x(zy)+(zy)x+y(zx)\big)$ &$=$\\
\multicolumn{4}{r}{$\delta \big( \big[[x,z]_1,y\big]_1 + \big[x,[y,z]_1\big]_1 \big).$}
      \end{longtable}
\noindent Hence, $\big(L, [\cdot, \cdot]_1\big)$ is a $\delta$-Lie algebra.
If $\delta \neq -\frac 12,$ then by Lemmas \ref{3} and \ref{5}, $\big(L, [\cdot, \cdot]_1\big)$ is a  $2$-step nilpotent Lie algebra.
\end{proof}

\section{$\delta$-Zinbiel algebras}\label{dZ}
 
Let us construct multilinear base elements of degree $3$ for a free $\delta$-Leibniz algebra $\big( \delta\neq 0\big)$. Below, we give a presentation of 6 non-base elements of degree 3 as a linear
combination of the base elements of degree 3:

\begin{longtable}{rclrcl}
$a (b  c) $&$ =$&$ \delta^{-1} (a b) c - (a  c) b,$&
$a (c  b) $&$ =$&$ \delta^{-1} (a c) b - (a  b) c,$\\ 

$b (a  c) $&$ =$&$ \delta^{-1} (b a) c - (b  c) a,$&
$b (c  a) $&$ =$&$ \delta^{-1} (b c) a - (b  a) c,$\\ 

$c (b  a) $&$ =$&$ \delta^{-1} (c b) a - (c  a) b,$&
$c (a  b) $&$ =$&$ \delta^{-1} (c a) b - (c  b) a.$ 

\end{longtable}

Following the approach in \cite{GK94}, we compute the dual operad $\delta$-$\mathcal{L}eib^!$, where the operad $\delta$-$\mathcal{L}eib$ governs the variety of $\delta$-Leibniz algebras. Then,
 
\begin{longtable}{rcl}
$\big[[a \otimes x, b \otimes y], c \otimes z\big] $&$+$&
$\big[[b \otimes y, c \otimes z], a \otimes x\big] \ +\  
\big[[c \otimes z, a \otimes x], b \otimes y\big]$ \\
&$=$&  $((a  b)  c) \otimes ((x \bullet y) \bullet z) - (c  (a  b)) \otimes (z \bullet (x \bullet y))$ \\
&&$\quad - ((b  a)  c) \otimes ((y \bullet x) \bullet z) + (c  (b  a)) \otimes (z \bullet (y \bullet x))$ \\
&&$\quad + ((b  c)  a) \otimes ((y \bullet z) \bullet x) - (a  (b  c)) \otimes (x \bullet (y \bullet z))$ \\
&&$\quad - ((c  b)  a) \otimes ((z \bullet y) \bullet x) + (a  (c  b)) \otimes (x \bullet (z \bullet y))$ \\
&&$\quad + ((c  a)  b) \otimes ((z \bullet x) \bullet y) - (b  (c  a)) \otimes (y \bullet (z \bullet x))$ \\
&&$\quad - ((a  c)  b) \otimes ((x \bullet z) \bullet y) + (b  (a  c)) \otimes (y \bullet (x \bullet z))\ =$\\

\multicolumn{3}{l}{$\mbox{
\big(
by using the above presentation of the 6 non-basis elements of degree 3}$}\\ 
\multicolumn{3}{r}{$\mbox{
as a linear combination of basis elements, we have \big)}$}\\

&$=$& 
$((a  b)  c) \otimes \big(-\delta^{-1} x \bullet (y \bullet z)- x \bullet (z \bullet y) +(x \bullet y ) \bullet z\big) +$\\
&& 
$((b  a)  c) \otimes \big(\delta^{-1} y \bullet (x \bullet z)+ y \bullet (z \bullet x) -(y \bullet x ) \bullet z\big) +$\\
&& $((c  a)  b) \otimes \big(-\delta^{-1} z \bullet (x \bullet y)-z \bullet (y \bullet x) +(z \bullet x ) \bullet y\big) +$\\

&& $((c  b)  a) \otimes \big(\delta^{-1} z \bullet (y \bullet x)+z \bullet (x \bullet y) -(z \bullet y) \bullet x\big) +$\\

&& $((b  c)  a) \otimes \big(-\delta^{-1} y \bullet (z \bullet x)-y \bullet (x \bullet z) +(y \bullet z) \bullet x\big) +$\\
&& $((a  c)  b) \otimes \big(\delta^{-1} x \bullet (z \bullet y)+x \bullet (y \bullet z) -(x \bullet z) \bullet y\big).$

\end{longtable}
Hence, the dual  operad $\delta$-$\mathcal{L}eib^!$ is  equivalent to the following:
\begin{longtable}{rcl}
$\delta \big  ( (x \bullet y) \bullet z - x \bullet (z \bullet y) \big) = x \bullet(y \bullet z).$
\end{longtable}

\subsection{$\delta$-Zinbiel algebras}

\begin{definition}\label{deltaLeibniz}
Let $Z$ be an algebra and $\delta$ be a fixed complex number. Then   $Z$ is called a  $\delta$-Zinbiel   algebra if the following identity holds true:
    \begin{equation} \label{dZIdentity}
      x(yz) \ = \ \delta \big(  (xy)z - x(zy) \big).
\end{equation}
\end{definition}

Let $\mathcal{Z}inb_{\delta}$ be a variety of $\delta$-Zinbiel algebras for some $\delta$. Then $\mathcal{Z}inb_0$ is the variety of  $2$-step left nilpotent algebras, and $\mathcal{Z}inb_1$ is the variety of Zinbiel  algebras.
Hence, the varieties of $\delta$-Zinbiel   algebras give a type of continuous deformation 
of the variety of Zinbiel algebras to the variety of  $2$-step left nilpotent algebras and vice versa.

\medskip

The classification of two-dimensional algebras over an algebraically closed field was established in \cite{kv16}. We present a classification of two-dimensional $\delta$-Zinbiel algebras, for which direct computations can be achieved.

\begin{theorem}\label{cldz}
Let $Z$ be a nontrivial $2$-dimensional $\delta$-Zinbiel algebra, then it is isomorphic to one of the algebras listed below

\begin{longtable}{|l|c|l|l|l|l|}

\hline
 ${\bf A}_3$ &$\delta$ & $e_1e_1= e_2$&  
 $e_1e_2= 0$ & $e_2e_1=0 $& $e_2e_2=0$ \\
\hline

\hline
 ${\bf B}_2(0)$ &$0$ & $e_1e_1= 0$&  
 $e_1e_2= e_1$ & $e_2e_1=0 $& $e_2e_2=0$ \\
\hline

\end{longtable}
\noindent Here: 
notations in  the first column are  given from {\rm \cite{kv16}} and 
the number $\gamma$ in the second column means that it is a $\gamma$-Zinbiel algebra. 
\end{theorem}

\begin{proposition}\label{nilpdZ}
    Let $Z$ be an $n$-dimensional nilpotent $\delta$-Zinbiel algebra $\big( \delta \notin \{ 0,1\} \big)$,
    then the nilpotency index of $Z$ is not greater than $n+1.$
    If the nilpotency index of $Z$ is equal to $n+1,$ then $\delta=-1$ and $Z$ is one-generated.
\end{proposition}

\begin{proof}
 It is easy to see that each product of $k$ elements in a $\delta$-Zinbiel algebra can be written as
 a sum of products with right-ordered brackets.
 Hence,  
 \begin{center}
     $ \sum\limits_{i+j=k} Z^iZ^j \ =: \ 
     Z^k \subseteq \underbrace{Z\big(\ldots (Z(ZZ))\ldots\big)}_{k \  times}\ =:\ Z^{k]}.$ 
\end{center}
Obviously, 
$ 0 = Z^{n+1]} \subseteq Z^{n-1]} \subseteq \ldots \subseteq Z^{2]} \subseteq Z^{1]}=Z.$
Hence, the nilpotency index of $Z$ is not greater than $n+1$. 
 Thanks to \cite[Theorem 3.1]{MO14}, there is only one $n$-dimensional $\delta$-Zinbiel algebras with the nilpotency index $n+1,$ i.e., a one-generated $(-1)$-Zinbiel algebra.
 Hence the nilpotency index of $L$ for $\delta \neq -1$ is not greater than $n.$
\end{proof}

\begin{corollary}\label{lengthZ}
    The length\footnote{For the definition of the length of algebras see, Definition \ref{lengthdef}.} of an $n$-dimensional  $\delta$-Zinbiel algebra  $Z$ $\big(\delta \neq 0,1\big)$
  is not greater than $n.$ 
      If the length  of $Z$ is equal to $n,$ then $\delta=-1$ and $Z$ is a nilpotent one-generated algebra.
\end{corollary}

\begin{proof}
    Let $Z$ be an $n$-dimensional $\delta$-Zinbiel algebra $\big(\delta \neq 0,1\big)$, such that $l({\rm A})=k.$
    Obviously\footnote{The idea is completely similar to the proof of \cite[Lemma 38]{KKL23}.}, there exists an $n$-dimensional nilpotent $\delta$-Zinbiel algebra $Z_0$, such that $l({Z_0})=k.$
    Thanks to Proposition  \ref{nilpdZ}, $k\le n$, and
    $k=n$ if and only if $\delta=-1.$ 
    As we will see from Theorem \ref{naz} $\big($see below$\big),$ 
    each finite-dimensional anti-Zinbiel algebra is nilpotent.
The statement is proved. \end{proof}

\begin{proposition}\label{34}
    Let $Z$ be a $\delta$-Zinbiel algebra, 
    then $Z$ is Lie admissible if and only if $\delta =\frac 12$ or it satisfies 
        $\sum\limits_{\sigma \in \mathbb S_3}  (-1)^{\sigma}x_{\sigma(1)}  \big(x_{\sigma(2)} x_{\sigma(3)}\big) \ =\ 0.$
\end{proposition}

\begin{proof}
    It is easy to see that $Z$ satisfies

    \begin{center}
        $\big(x,y,z\big) \ =\   x(yz)+ (\delta^{-1}-1)x(zy).$
    \end{center}
Hence $\big(x,y,z\big) - \big(x,z,y\big) = (2-\delta^{-1}) \big(x(zy)-x(yz) \big).$
It follows, 
\begin{center}
$\sum\limits_{\sigma \in \mathbb S_3} (-1)^{\sigma}\big (x_{\sigma(1)},x_{\sigma(2)}, x_{\sigma(3)} \big)\ =\  
(2-\delta^{-1}) \sum\limits_{\sigma \in \mathbb S_3}  (-1)^{\sigma}x_{\sigma(1)}  \big(x_{\sigma(2)} x_{\sigma(3)}\big),$
\end{center}
 where the left side gives the Lie-admissible identity.
Hence, our statement is proved. \end{proof}

\begin{proposition}
        Let $Z$ be a $\delta$-Zinbiel  algebra $\big(\delta \neq 0\big).$ 
    Then $Z$ is a $\big(\delta^{-1}-1\big)$-right-symmetric algebra. 
\end{proposition}

\begin{proof}
It is easy to see that
\begin{longtable}{lcl}
$\delta(xy)z$&$=$&$ x(yz)+\delta x(zy),$\\ 
$\delta(xz)y$&$=$&$ x(zy)+\delta x(yz).$\\ 
\end{longtable}
Subtracting these two identities, we obtain the following identity:
\begin{center}
    $\delta(xy)z + (\delta-1)x(yz)\ =\ \delta(xz)y+(\delta-1)x(zy).$
    \end{center}
Hence, each $\delta$-Zinbiel algebra is a $\big(\delta^{-1}-1\big)$-right-symmetric algebra.
\end{proof}

\begin{definition}
Let $C(a) = \big\{x\in L \ |\  ax =xa=0 \big\}$ be the annihilator of $a\in L.$
\end{definition}

\begin{lemma}\label{centre} 
Let  $Z$ be a $\delta$-Zinbiel algebra $\big( \delta \notin \{ 0,-1 \}\big).$ 
Then $C(a)\subseteq C(a^2)$ for all $a\in Z$.
\end{lemma}

\begin{proof} For any $x\in C(a)$ we have 
\begin{longtable}{rcccccl}
$(aa)x $&$=$&$ \delta^{-1} a(ax)+a(xa) $& $=$&$0$&$\Longrightarrow$&$ a^2x=0,$\\  
$0 $&$=$&$ (xa)a $& $=$&$\delta^{-1}x(aa)+x(aa)  $&$\Longrightarrow$&$  xa^2=0.$\\ 
\end{longtable}\noindent 
It means $x\in C(a^2)$ and  $C(a)\subseteq C(a^2)$ for all $a\in L.$
\end{proof} 

\begin{lemma}\label{centre} 
Let  $Z$ be a $\delta$-Zinbiel algebra $\big(\delta\notin \{-1,0,1\} \big)$. 
Then $C(a)$ is a subalgebra of $Z$.
\end{lemma}

\begin{proof} For any $x,y\in C(a)$ we have 
\begin{longtable}{rcccl}
$\delta (ax)y $&$=$&$  a(xy)+ \delta  a(yx) $&$\Longrightarrow$&$ \delta a(xy)+a(yx)=0,$\\  
$\delta (ay)x $&$=$&$   a(yx)+ \delta  a(xy) $&$\Longrightarrow$&$ \delta a(yx)+a(xy)=0,$\\ 
$\delta (xy)a$&$=$&$   x(ya)+\delta  x(ay)$&$\Longrightarrow$&$ (xy)a=0.$\\ \end{longtable}
\noindent From the first and second relations, we have $(1-\delta^2)a(xy)=0$. Hence, $xy\in C(a).$
\end{proof} 

\begin{lemma}\label{eigenvalue} 
Let $Z$ be a $\delta$-Zinbiel algebra $\big(\delta\neq-1,0\big)$. 
If $v$ is an eigenvector of the linear operator $L_a$ with eigenvalue $\mu,$ 
then if $v^2\neq 0,$ then  $L_a(v^2) = \frac{\mu\delta}{\delta+1}v^2.$
\end{lemma}

\begin{proof} 
We have $av =\mu v,$ then  $(av)v = \mu v^2$ and $(av)v = (\delta^{-1}+1) av^2.$
Hence, $L_a(v^2) = \frac{\mu\delta}{\delta+1}v^2.$ \end{proof} 

\begin{proposition} Let $Z$ be a $\delta$-Zinbiel algebra $\big(\delta\neq-1,0\big)$. Then if there exists $a\in Z,$ such that  $L_a$ has an eigenvector $v$ with a non-zero eigenvalue, such that $v^2v\neq0,$ then $\delta=1.$
\end{proposition}

\begin{proof} 
If $av =\mu v,$ then by Lemma \ref{eigenvalue} we have  $av^2 = \frac{\mu\delta}{\delta+1}v^2.$ It is easy to see $vv^2=   \frac{\delta}{\delta+1}v^2v .$

\begin{longtable}{rcccccccl}
$\frac{\mu\delta}{\delta+1} v^2v$&$=$&$\mu vv^2$&$=$&$ (av)v^2$&$=$&
$ \delta^{-1} a(vv^2)+a(v^2v)$& $=$& 
$ \frac{\delta+2}{\delta+1} a(v^2v) $  \\

\end{longtable}

\begin{longtable}{rcccccccl}
$ \frac{\mu\delta}{\delta+1} v^2v $ &$=$&
$(av^2)v $&$=$&$ \delta^{-1} a(v^2v)+a(vv^2) $&
$=$&$  \frac{\delta^2+\delta+1}{\delta(\delta+1)}a(v^2v).$
\end{longtable}
\noindent Analyzing the last relations, we have 
$\delta+2 \ = \ \delta+1+\delta^{-1},$ which gives $\delta=1.$
\end{proof}


 \subsection{Anti-Zinbiel algebras}\label{antiz}

Thanks to Proposition \ref{nilpdZ}, anti-Zinbiel algebras\footnote{Our definition of anti-Zinbiel algebras introduced via anti-Leibniz algebras is different from anti-Zinbiel algebras introduced via anti-dendriform algebras in \cite{LB23}.}, 
i.e., $\big(-1\big)$-Zinbiel algebras, have a essencial role in all $\delta$-Zinbiel algebras. The present subsection is dedicated to a more detailed study of anti-Zinbiel algebras. 
Let us remember that each Zinbiel algebra is right-commutative. The following statement gives an anti-analog for anti-Zinbiel algebras.

\begin{proposition}
    Each anti-Zinbiel algebra is an anti-right-commutative algebra.
\end{proposition}
\begin{proof}
    We have 
    $(xy)z \ =\  x(zy-yz) \ = \ -\ x (yz-zy) \ = \ -\  (xz)y.$
\end{proof}

 \begin{lemma}\label{solvableAZ}
     Each finite-dimensional anti-Zinbiel algebra is solvable.
 \end{lemma}

 \begin{proof}
     Our proof is based on the original result by Dzhumadildaev and Tulenbaev  \cite{dk05}, adapted for the anti-Zinbiel case.
Let $Z$ be a finite-dimensional anti-Zinbiel algebra 
and\begin{center}
$C_r\big(a\big) = \big\{x \in Z\ | \ ax\ =\ 0 \big\}.$ \end{center}
     \begin{enumerate}
          \item[Step 1.]
          If $x \in C_r\big(a\big),$ then $(ab)x= -(ax)b=0,$ and $C_r\big(a\big) \subseteq C_r\big(ab\big)\footnote{This step is valid only for Zinbiel and anti-Zinbiel algebras, and it is not valid in the case of an arbitrary $\delta$-Zinbiel algebra $\big(\delta\neq \pm1 \big)$.}.$

     \item[Step 2.] If $a \in Z$ and $L_a\neq 0,$ then
     there exist $b \in Z,$ such that $ab=\lambda b,$ for   $\lambda \in \mathbb C\setminus \{0\}.$ Hence, $0\ =\ (ab)b \ = \ \lambda b^2,$ i.e., there exists $b,$ such that $b^2=0.$

     \item[Step 3.] Take an element $a_1\in Z.$ If there is an element $a_1 \in Z,$ such that $C_r\big(a_1\big) \neq C_r\big(a_1a_2\big),$ then, due to {\bf Step 1}, $C_r\big(a_1\big) \subset C_r\big(a_1a_2\big).$ Now we repeate the procedure for $a_3:$ if 
      $C_r\big(a_1a_2\big) \neq C_r\big((a_1a_2)a_3\big),$ then,  $C_r\big(a_1\big) \subset C\big(a_1a_2\big) \subset C_r\big((a_1a_2)a_3\big),$ and so on.
      Finally,  we obtain that a sequence of nonzero elements $a_1, \ldots, a_k \in Z,$ such that 
      \begin{center}
           $C_r\big(a_1\big) \subset C\big(a_1a_2\big) \subset C_r\big((a_1a_2)a_3\big) \subset
          C_r\big( (\ldots ((a_1a_2)a_3) \ldots )a_k\big) \subseteq Z.$
      \end{center}
      Since $Z$ is finite-dimensional, this sequence terminates at some $k,$ 
      i.e., there is an element $a_0 =(\ldots ((a_1a_2)a_3) \ldots )a_k,$ such that for each $b\in Z,$ 
      we have $C_r\big(a_0\big)\ =\  C_r\big(a_0b \big).$

\item[Step 4.] If $L_{a_0}=0,$ then $C_r(a_0)=Z.$
If  $L_{a_0} \neq 0,$ then, due to {\bf Step 2}, there exists $b\in Z,$ such that $b^2=0$ and $a_0b=\lambda b$ for $\lambda \neq 0.$ But    $b\in C_r(a_0 b) = C_r(a_0),$ hence $a_0b=0$ and $\lambda =0,$ hence we have a contradiction and $C_r(a_0)=Z.$

\item[Step 5.] Let us consider $I=Za_0 = \big\{ xa_0  |  x \in Z \big\}.$ 
We have $(xa_0)y  =  -   (xy)a_0 \in I$ and $a_0y=0$ for each $y\in Z.$
Therefore, 
$y ( x a_0)  =  y(xa_0-a_0x) =  (y x) a_0 \in I.$ 
Hence, $I$ is an ideal of $Z.$

\item[Step 6.]
By taking a basis $\big\{e_1=a_0, e_2, \ldots, e_n \big\}$ in $Z,$ we have that 
$e_1\in C_r(a_0)=Z$ and $e_1e_1=0.$
Hence, $I \ = \ \big\langle e_1e_1=0, e_2e_1, \ldots, e_ne_1 \big\rangle,$ 
i.e.,  ${\rm dim} \ I < {\rm dim } \ Z.$
It gives that if $I\neq 0,$ then $I$ is a proper ideal of $Z.$
If $I=0,$ then we can take $I:=\big\langle a_0 \big\rangle.$

\item[Step 7.]
Obviously, if ${\dim } \ Z=1,$ then $Z$ has trivial multiplication and, in particular, it is solvable.
We will use the induction on ${\rm } \ Z.$
Assume that each $(k-1)$-dimensional anti-Zinbiel algebra is solvable.
By {\bf Step 6}, $Z$ has some proper ideal $I,$ 
such that ${\rm dim} \ I<n$ and  ${\rm dim} \ Z\big/I <n.$ 
By the induction hypothesis 
$I$ and $Z\big/I$ are solvable. Hence, $Z$ is solvable.     \end{enumerate} \end{proof}

 \begin{theorem}\label{naz}
     Each finite-dimensional anti-Zinbiel algebra is nilpotent.
 \end{theorem}

 \begin{proof}
     Our proof is based on the original result by Towers   \cite{T23}, adapted for the anti-Zinbiel case.
Let $Z$ be a finite-dimensional anti-Zinbiel algebra, then by Lemma \ref{solvableAZ}, $Z$ is solvable.
     \begin{enumerate}
          \item[Step 1.] Let $B$ be a right ideal of $Z,$ then
\begin{center}
   $ Z\big(ZB\big) \subseteq \big(ZZ\big)B+Z\big(BZ\big) \subseteq ZB$ and 
   $\big(ZB\big)Z \subseteq \big(ZZ\big)B \subseteq ZB,$
\end{center}
i.e., $ZB$ is an ideal of $Z.$

\item[Step 2.] 
Let $A$ be a minimal ideal of $Z$ and $B$ be a minimal right ideal of $Z,$ such that $B \subseteq A.$    
  Clearly, $ZB \subseteq ZA \subseteq A,$ so, due to {\bf Step 1},  $ZB=0$ or $ZB=A.$
\begin{enumerate}
    \item[{\rm (I)}]   If $ZB=0,$ then $B$ is an ideal of $Z,$ so $B=A.$
   \item[{\rm (II)}]   If $ZB=A,$ then  $BZ^2$ is a right ideal inside $B.$
      If $B \subseteq BZ^{(k)}$ we have
   \begin{center}
       $B \subseteq \big(BZ^{(k)}\big)Z^{(k)} \subseteq BZ^{(k+1)}$
   \end{center}
and due to sovability of $Z,$ we have $BZ^2=0.$
   If $BZ=B,$ then \begin{center}
       $B\ =\ \big(BZ\big)Z \subseteq BZ^2=0.$
   \end{center} Hence, $BZ=0.$
    Now, 
    \begin{center}
        $Z\big(Z^2B\big) \subseteq \big(Z^2Z\big)B+Z\big(BZ^2\big) \subseteq Z^2B$
        and $\big(Z^2B\big)Z \subseteq \big(Z^2Z\big)B \subseteq Z^2B,$ 
    \end{center}    so $Z^2B$ is an ideal if $Z$ and $Z^2B$ is inside $A.$
    Hence, 
    \begin{enumerate}
        \item[(II.a)]   if  $Z^2B=A,$ then $B \subseteq Z^2B.$
        But if $B \subseteq Z^{(k)}B,$ then $B \subseteq Z^{(k)}\big(Z^{(k)}B\big) \subseteq Z^{(k+1)}B$
        and due to the solvability of $Z,$ we have a contradiction.
        
 \item[(II.b)] if     $Z^2B=0,$ then 
 \begin{center}
     $AZ\ =\ \big(ZB\big)Z \ \subseteq\  Z^2B\ =\ 0$ and 
 $ZA \ =\  Z\big(ZB\big) \ \subseteq \ Z^2B+Z\big(BZ\big)\ =\ 0,$ 
 \end{center}
i.e., $A \subseteq {\rm Ann} \ Z$ and it follows that ${\rm dim} \ A=1$ and $B=A.$
    
    \end{enumerate}
    
    \end{enumerate}

\item[Step 3.] 
Let $A$ now a minimal ideal of $Z,$
then $AZ$ is a right ideal of $Z$ inside $A.$
Due to the {\bf Step 2},  $AZ\ =\ A$ and 
$A\ =\ AZ\  =\  \big(AZ\big)Z \subseteq AZ^2.$ 
As befoure if $A \subseteq AZ^{(K)},$ then $A \subseteq \big(AZ^{(k)}\big)Z^{(k)} \subseteq AZ^{(k+1)}$ 
and due to the solvability of $Z,$ we have $AZ=0.$
Similarly, $ZA=0$ and, due to the minimality of $A$, we have ${\rm dim} \ A=1.$

\item[Step 4.] 
    Obviously, if ${\dim } \ Z=1,$ then $Z$ has trivial multiplication and, in particular, it is nilpotent.
We will use the induction on ${\rm } \ Z.$
Assume that each $(k-1)$-dimensional anti-Zinbiel algebra is nilpotent.
By {\bf Step 3}, $Z$ has a minimal ideal $A,$ 
such that ${\rm dim} \ A=1$ and  ${\rm dim} \ Z\big/A =n-1.$ 
By the induction hypothesis 
$A$ and $Z\big/A$ are nilpotent and there exist $n,$ such that $Z^n \subseteq A.$ 
Hence, $Z^{n+1} \subseteq Z^nA =0$ and $Z$ is nilpotent.

    \end{enumerate}    
 \end{proof}

\begin{lemma}
    Maximal subalgebras of  anti-Zinbiel algebras are ideals.
\end{lemma}
\begin{proof}
    It is similar to \cite[proof of the Corollary 2.5]{T23}.
\end{proof}

\section{Symmetric $\delta$-Leibniz admissible algebras}

\begin{definition}
   An algebra ${\rm A}$ is called a symmetric $\delta$-Leibniz admissible algebra
    if $\big({\rm A}, \ [\cdot,\cdot]_{\delta} \big)$ is a symmetric $\delta$-Leibniz algebra.  
 
\end{definition}

Let $\mathcal{S}LA_{\delta}$ be a variety of symmetric $\delta$-Leibniz admissible  algebras  for some $\delta$. Then 
$\mathcal{S}LA_1$ is the variety of Lie admissible  algebras,  
$\mathcal{S}LA_0$ is the variety of symmetric Leibniz    algebras, and 
$\mathcal{S}LA_{-1}$ is the variety of Jacobi-Jordan admissible  algebras.
Hence, the varieties of symmetric $\delta$-Leibniz admissible algebras give a type of continuous deformation  of the variety of Lie admissible algebras to the variety of Jacobi-Jordan admissible algebras
via symmetric Leibniz algebras and vice versa.

\begin{theorem}
    Let ${\rm A}$ be a symmetric $\delta$-Leibniz admissible algebra $\big( \delta \neq 0\big),$ 
    then ${\rm A}$ satisfies the following identities: 
	\begin{equation}\label{idsdladd1}	
	(x,y,z)_{\delta} +\delta^2(z,x,y)_{\delta^{-1}}+\delta^2(y,z,x)_{\delta} \ =\  
    \delta\big( (y,x,z)_{\delta}+(x,z,y)_{\delta}+\delta^2 (z,y,x)_{\delta^{-1}} \big),	 
		\end{equation}
	\begin{equation}\label{idsdladd2}	
	(x,y,z)_\delta+(z,x,y)_{\delta^{-1}} +\delta^2(y,z,x)_{\delta^{-1}} \ =\  
    \delta\big( (y,x,z)_{\delta}+(x,z,y)_{\delta^{-1}}+(z,y,x)_{\delta^{-1}} \big).	 
		\end{equation}
  
\end{theorem}

\begin{proof}
   As $\big({\rm A}, \  [\cdot, \cdot]_\delta \big)$ is a symmetric $\delta$-Leibniz algebra, then it   satisfies:

   \begin{longtable}{lcl}

$\big[[x,y]_\delta,z \big]_\delta$ & $=$ & $\delta\big( \big[[x,z]_\delta,y\big]_\delta + \big[x,[y,z]_\delta\big]_\delta\big),$\\

$\big[z,[x,y]_\delta\big]_\delta$ & $=$ & $\delta\big( \big[[z,x]_\delta,y\big]_\delta + \big[x,[z,y]_\delta\big]_\delta\big).$

   \end{longtable}\noindent
The last, consequently, gives 
\begin{longtable}{lcl}
  $(xy)z-\delta z(xy)-\delta (yx)z+\delta^2 z(yx)$ &$=$&$\delta (xz)y -\delta^2 y(xz) -\delta^2 (zx)y+\delta^3 y(zx)+$\\
  &&$\delta x(yz)-\delta^2 (yz)x -\delta^2 x(zy)+\delta^3 (zy)x,$\\

  $z(xy)-\delta (xy)z-\delta z(yx)+\delta^2 (yx)z$ &$=$&$\delta (zx)y -\delta^2 y(zx) -\delta^2 (xz)y+\delta^3 y(xz)+$\\
  &&$\delta x(zy)-\delta^2 (zy)x -\delta^2 x(yz)+\delta^3 (yz)x.$  
\end{longtable}\noindent 
Hence, 

\begin{longtable}{rcl}
$(x,y,z)_{\delta} +\delta^2(z,x,y)_{\delta^{-1}}+\delta^2(y,z,x)_{\delta} $&$ =$&$  
    \delta\big( (y,x,z)_{\delta}+(x,z,y)_{\delta}+\delta^2 (z,y,x)_{\delta^{-1}} \big),$\\	 
	$(x,y,z)_\delta+(z,x,y)_{\delta^{-1}} +\delta^2(y,z,x)_{\delta^{-1}} $&$=$&$  
    \delta\big( (y,x,z)_{\delta}+(z,y,x)_{\delta^{-1}}+(x,z,y)_{\delta^{-1}} \big).$	 
\end{longtable}
\end{proof}

It is known that symmetric $\big(-1\big)$-Leibniz admissible $\big(=$ Jacobi-Jordan admissible$\big)$ algebras  satisfy the identity $x^2x=-xx^2.$
In the case of symmetric $1$-Leibniz admissible $\big(=$ Lie admissible$\big)$ algebras, 
power-associative algebras also play an important role \cite{B84}.

\begin{theorem}\label{nilds}
     Let ${\rm A}$ be a symmetric $\delta$-Leibniz admissible algebra. 
     Then
     \begin{enumerate}
         \item[$({\rm A})$] if $\delta\notin \big\{ \pm 1,\frac12 \big\}$ then ${\rm A}$ is a nilalgebra with nilindex $3;$  
         \item[$({\rm B})$] if $\delta= \frac 12,$ then ${\rm A}$ is a  power-associative algebra. 
     \end{enumerate}
 \end{theorem}

 \begin{proof}  
  Let ${\rm A}$ be a symmetric $\delta$-Leibniz admissible algebra and   $\delta \notin \big\{  \pm 1,\frac{1}{2} \big\}$. Then, according to Theorem \ref{symmetric}, we have 
\begin{center}
    $\big[[x,x]_\delta,[x,x]_\delta\big]_\delta\ =\ 
    \big[[x,x]_\delta,x \big]_\delta \ =\ \big[x,[x,x]_\delta\big]_\delta \ =\ 0.$
\end{center}
From this, we obtain the following
\begin{longtable}{lclclcl}
$(1-\delta)^3x^2x^2$&$=$& $(1-\delta)(x^2x-\delta xx^2)$&$=$&$ (1-\delta)(xx^2-\delta x^2x)$&$=$&$0.$	 
\end{longtable}

\noindent Then 
    $ x^2x^2 \ = \ x^2x\ =\ xx^2\ =\ 0.$ 
    Hence, due to  \cite[Lemma 3]{Albert}, we have  $x^n=0$ for all $n>2.$ 

\medskip 

         Let      ${\rm A}$ be a symmetric $\frac{1}{2}$-Leibniz admissible algebra.
         Then, according to Theorem \ref{symmetric}, we have 
$\big[[x,x]_\frac{1}{2},\ x\big]_\frac{1}{2}\ =\ \big[x,\ [x,x]_\frac{1}{2}\big]_\frac{1}{2}.$
From this, we obtain that $x^2x=xx^2$, and now we  consider 
\begin{longtable}{lclcl}
$\big[[x,x]_\frac{1}{2},\ [x,x]_\frac{1}{2} \big]_\frac{1}{2}$&$=
$&$\big[[x,[x,x]_\frac{1}{2}]_\frac{1}{2},\ x\big]_\frac{1}{2}$&
$=$&$\big[x,\ [[x,x]_\frac{1}{2},x]_\frac{1}{2}\big]_\frac{1}{2}.$
\end{longtable}
Hence, we obtain the following
\begin{longtable}{lclcl}
$\frac{1}{8}x^2x^2$&$=$&
$\frac{1}{4}\big(x^3x-\frac{1}{2}xx^3\big)$&$=$&
$\frac{1}{4}\big(xx^3-\frac{1}{2}x^3x\big).$
\end{longtable}
Therefore, we obtain 
$x^2x^2\ =\ x^3x\ =\ xx^3.$
Then, the   algebra ${\rm A}$ satisfies the conditions from \cite[Lemma 3]{Albert} and hence, ${\rm A}$ is power-associative.  
 
\end{proof}

\subsection{Algebras of $\delta$-biderivation-type}

\begin{definition}
	{  An algebra ${\rm A}$ is called an algebra of $\delta$-biderivation-type $\big(\delta{\rm BD}$-algebra$\big)$ if it satisfies:
		\begin{equation}\label{Def1}	
		x [y, z]_{\delta}= \delta \big( [x y,  z]_{\delta}+[y,x z]_{\delta} \big),
		\end{equation}
		\begin{equation}\label{Def2}
		[y, z]_{\delta} x=\delta \big( [y x,  z]_{\delta}+[y, z x]_{\delta}\big).
		\end{equation}
  Which equivalent to say that, for any $x$ element of ${\rm A},$ 
   ${\rm L}_{x}$ and ${\rm R}_{x}$ are $\delta$-derivations of the algebra $({\rm A}, \ [\cdot,\cdot]_{\delta})$, where ${\rm L}_{x}$ and ${\rm R}_{x}$ denote, respectively, the left and right multiplications by the element $x \in {\rm A}$ in the algebra ${\rm A}$.  
				
	}
\end{definition}	

Let $\mathcal{BD}_{\delta}$ be a variety of algebras of  $\delta$-biderivation-type for  some $\delta$. Then 
$\mathcal{BD}_1$ is the variety of biderivation-type  algebras,  
$\mathcal{BD}_{-1}$ is the variety of anti-biderivation-type    algebras.
Hence, the varieties of algebras of  $\delta$-biderivation-type give a type of continuous deformation  of the variety of algebras of  biderivation-type to the variety of algebras of  anti-biderivation-type
 and vice versa.

\begin{proposition}
    
    Let ${\rm A}$ be a $\delta{\rm BD}$-algebra, 
    then ${\rm A}$ is a symmetric $\delta$-Leibniz admissible algebra.
 
\end{proposition}

\begin{proof}  Let ${\rm A}$ be a $\delta{\rm BD}$-algebra, then identities \eqref{Def1} and \eqref{Def2} hold. By considering the expressions $\eqref{Def1} - \delta\cdot \eqref{Def2}$ and $\eqref{Def2} - \delta\cdot\eqref{Def1}$, we obtain the following identities, respectively:

\begin{longtable}{lcl}
$\big[x,[y,z]_\delta\big]_\delta$ & $=$ & $\delta\big( \big[[x,y]_\delta,z\big]_\delta + \big[y,[x,z]_\delta\big]_\delta\big),$\\

$\big[[y,z]_\delta,x\big]_\delta$ & $=$ & $\delta\big( \big[[y,x]_\delta,z\big]_\delta + \big[y,[z,x]_\delta\big]_\delta\big).$
   \end{longtable}
\end{proof}

\begin{proposition} Every $\delta BD$-algebra $\big(\delta\neq1 \big)$ is a nilalgebra with nilindex $3.$ In particular, each $\delta BD$-algebra is a power-associative algebra.
\end{proposition}

\begin{proof} Taking $x = y = z$ in the identities \eqref{Def1} and \eqref{Def2}, we have
$$\begin{array}{lcl}
(1-\delta)xx^2 & = & \delta(1-\delta)\big(x^2x + xx^2\big),\\[1mm]
(1-\delta)x^2x & = & \delta(1-\delta)\big(x^2x + xx^2\big), \end{array}$$
or $$\begin{array}{lcl}
xx^2 & = & \delta\big(x^2x + xx^2\big),\\[1mm]
x^2x & = & \delta\big(x^2x + xx^2\big), \end{array}$$ 
and from this we have $x^2x=xx^2=x^3=0.$ Hence, taking
elements $x^2$, $x$ and $x$ in the identities \eqref{Def1} and \eqref{Def2}, we have
$$(1-\delta)x^2x^2  = \delta(1-\delta)\big(x^3x + xx^3\big)=0,$$
i.e. $x^4 =x^3x=x^2x^2=xx^3=0.$ Then, it follows from \cite[Lemma 3]{Albert}, that ${\rm A}$ is power-associative
\end{proof}

 \subsection{Examples of $2$- and $3$-dimensional algebras.}

The classification of two-dimensional algebras over an algebraically closed field was established in \cite{kv16}. We present a classification of two-dimensional symmetric $\delta$-Leibniz admissible algebras and algebras of $\delta$-biderivation-type, for which direct computations can be achieved.

\begin{theorem}\label{cld23}
Let $\bf S$ be a nontrivial $2$-dimensional 
symmetric $\delta$-Leibniz admissible algebra or   algebra of $\delta$-biderivation-type, then it is isomorphic to one of the algebras listed below:

\begin{longtable}{|l|c|c|l|l|l|l|}

 \multicolumn{1}{l|}{}&${\mathcal S}LA_\delta $ &$\delta {\rm BD}$    \\

\hline
  ${\bf A}_1(\frac 12)$ & $-$ &$1$ & $e_1e_1= e_1+e_2$&  
 \multicolumn{2}{|c|}{$e_1e_2\ =\  \frac 12 e_2\ =\  e_2e_1$}& $e_2e_2=0$ \\
\hline

\hline
  ${\bf A}_2$ & $-1$& $1$ & 
$e_1e_1=  e_2$&  
 $e_1e_2=   e_2$ & $e_2e_1=-  e_2 $& $e_2e_2=0$ \\
\hline
 
\hline
  ${\bf A}_3$ & $\delta$& $\delta$ & 
$e_1e_1=  e_2$&  
\multicolumn{2}{|c|}{ $e_1e_2\ =\   0\ = \ e_2e_1$}& $e_2e_2=0$ \\
\hline

 \hline
  ${\bf A}_4(0)$ & $-1$& $-$ & 
$e_1e_1=  e_2$&  
 $e_1e_2=   e_2$ & $e_2e_1=-  e_1 $& $e_2e_2=0$ \\
\hline

 \hline
  ${\bf B}_2(\alpha)$ & $-$& $1$ & 
$e_1e_1= 0$&  
 $e_1e_2= (1-\alpha) e_1$ & $e_2e_1=  \alpha   e_1 $& $e_2e_2=0$ \\
\hline

   \hline
 ${\bf B}_3 $ & $-1$& $\pm 1$ & 
$e_1e_1= 0$&  
 $e_1e_2=  e_2$ & $e_2e_1=  -   e_2 $& $e_2e_2=0$ \\
\hline

   \hline
 ${\bf C}(\frac 12,0) $ & $-$& $1$ & 
$e_1e_1= e_2$&  
 $e_1e_2= \frac 12 e_1$ & $e_2e_1=  \frac 12   e_2 $& $e_2e_2=e_2$ \\
\hline

   \hline
  ${\bf D}_1(\alpha,0) $ & $-$& $1$ & 
$e_1e_1= e_1$&  
 $e_1e_2= (1-\alpha) e_1$ & $e_2e_1=  \alpha   e_1 $& $e_2e_2=0$ \\
\hline

   \hline
  ${\bf D}_2(\alpha,\alpha) $ & $-$& $1$ & 
$e_1e_1= e_1$&  
\multicolumn{2}{|c|}{ $e_1e_2\ =\  \alpha e_2 \ = \ e_2e_1$}& $e_2e_2=0$ \\
\hline

   \hline
 ${\bf D}_2(1,1) $ & $\frac 12$& $\frac 12$ & 
$e_1e_1= e_1$&  
\multicolumn{2}{|c|}{ $e_1e_2\ =\  e_2\ =\  e_2e_1$}& $e_2e_2=0$ \\
\hline

 \hline
  ${\bf D}_2(0,0) $ & $\frac 12$& $\frac 12$ & 
$e_1e_1= e_1$&  
\multicolumn{2}{|c|}{ $e_1e_2\ =\  0\ =\  e_2e_1$}& $e_2e_2=0$ \\
\hline

 \hline
  ${\bf D}_3(0,0) $ & $-$& $1$ & 
$e_1e_1= e_1$&  
$e_1e_2 = e_1$ &$ e_2e_1=-e_1$ & $e_2e_2=0$ \\
\hline

 \hline
  ${\bf E}_1(\alpha,\beta,\alpha,\beta)  $ & $-$& $1$ & 
$e_1e_1= e_1$&  
\multicolumn{2}{|c|}{$e_1e_2= \alpha  e_1+\beta e_2=e_2e_1$}& $e_2e_2=e_2$ \\
\hline


 \hline
  ${\bf E}_1(0,0,0,0)  $ & $\frac 12$& $\frac 12$ & 
$e_1e_1= e_1$&  
 \multicolumn{2}{|c|}{$e_1e_2 \ =  \ 0\ = \ e_2e_1$}& $e_2e_2=e_2$ \\
\hline

\end{longtable}\noindent 
Here: 
notations in  the first column are  given from {\rm \cite{kv16}}; 
the number $\gamma$ in the second column means that it is a symmetric $\gamma$-Leibniz admissible algebra;
the number $\gamma$ in the third column means that it is an algebra of $\gamma$-biderivation-type; 
we also are not interested in symmetric $1$-Leibniz admissible algebras, 
because each $2$-dimensional algebra is it.
\end{theorem}

As we can see, nilalgebras play an important role in symmetric $\delta$-Leibniz admissible algebras 
and algebras of $\delta$-biderivation-type. 
A classification of $3$-dimensional nilalgebras was given in \cite{ks25}.
Using the present classification, we characterize $3$-dimensional nilagebras satisfying 
  symmetric $\delta$-Leibniz admissible  
and  $\delta$-biderivation-type identities.

\begin{longtable}{l|lcllllll}

$\delta{\rm BD}$ & $\mathfrak{g}_{1}$ & $:$ & 
 $e_2e_3 =e_1$ &  $e_3e_2 =-e_1$ \\
\hline

$({\pm 1}){\rm BD}$  & 
$\mathfrak{g}_2$ & $:$ & $e_1e_3 =e_1$ &  $e_2e_3=e_2$ &   $e_3e_1 =-e_1$ &  $e_3e_2=-e_2$\\
\hline

$({\pm 1}){\rm BD}$  & $\mathfrak{g}^{\alpha}_3$ & $:$ &  $e_1e_3 =e_1+e_2$ & $e_2e_3=\alpha e_2$ &
$e_3e_1 =-e_1-e_2$ & $e_3e_2=-\alpha e_2$\\
\hline

$({\pm 1}){\rm BD}$  & $\mathfrak{g}_4$ & $:$ &
$e_1e_2 =e_3$ & $ e_1e_3=-e_2$ & $e_2e_3=e_1$ &\\
&&& $e_2e_1 =-e_3$ & $ e_3e_1=e_2$ & $e_3e_2=-e_1$ \\
\hline

$({- 1}){\rm BD}$  &$\mathcal{A}_1^{\alpha}$ & $:$ &
$e_1e_2=e_3$ & $e_1e_3 =e_1+e_3$ & $e_2e_3=\alpha e_2$ &\\
&&&
$e_2e_1=-e_3$ & $e_3e_1 =-e_1-e_3$ & $e_3e_2=-\alpha e_2$\\
\hline

$({- 1}){\rm BD}$  &$\mathcal{A}_2$& $:$ &  
$e_1e_2=e_1$ & $e_2e_3=e_2$  &
$e_2e_1=-e_1$ & $e_3e_2=-e_2$  \\
\hline

$({- 1}){\rm BD}$  &$\mathcal{A}_3$ & 
$:$&$e_1e_2=e_3$ & $e_1e_3=e_1$  &$e_2e_3=e_2$  &\\ &&&
$e_2e_1=-e_3$ & $e_3e_1=-e_1$  &$e_3e_2=-e_2$  \\
\hline

${\delta}{\rm BD}$  &$\mathcal{N}_1$ & 
$:$&$e_1e_1=e_2$ \\
\hline

$\mathcal{S}LA_{\pm 1}$  & $\mathcal{N}_2$ & 
$:$&$e_1e_1=e_2$  &$e_1e_3=e_1$&$e_3e_1=-e_1$\\
\hline

$({\pm 1}){\rm BD}$ & $\mathcal{N}_3$ & 
$:$&$e_1e_1=e_2$  &$e_1e_3=e_3$&$e_3e_1=-e_3$\\
\hline

${\delta}{\rm BD}$  &$\mathcal{N}_4$ & 
$:$&$e_1e_1=e_2$  &$e_1e_3=e_2$&$e_3e_1=-e_2$ \\
\hline

$\mathcal{S}LA_{\pm 1}$  & ${\mathcal N}_{5}$ &$:$& $e_1 e_1 = e_2$  & $e_1 e_3=e_3$ & $e_3 e_1=-e_3$ & $e_3 e_3=e_2$\\
\hline

${\delta}{\rm BD}$  &${\mathcal N}_{6}^\alpha$ & $:$&  $e_1 e_1 = e_2$ & $e_1 e_3= \alpha e_2$ & $e_3 e_1= -\alpha e_2$ &  $e_3 e_3=  e_2$  \\
\hline

$({ 1}){\rm BD}$  &$\rm{N}_1$ &   $:$&     $e_1 e_1 = e_2$ & $e_2 e_1=e_3$ &  \\
\hline

$({ 1}){\rm BD}$ 
&$\rm{N}_2^\alpha$&$:$& $e_1 e_1 = e_2$ & $e_1 e_2=e_3$ & $e_2 e_1=\alpha e_3$  \\ 
\hline

$({\frac 12}){\rm BD}$
&$\rm{N}_2^1$&$:$& $e_1 e_1 = e_2$ & $e_1 e_2=e_3$ & $e_2 e_1= e_3$  \\ 
\hline

$\mathcal{S}LA_{-1}$ 
&$\rm{N}_2^{-1}$&$:$& $e_1 e_1 = e_2$ & $e_1 e_2=e_3$ & $e_2 e_1=- e_3$  \\ 
\hline

$({1}){\rm BD}$  &
$\bf{N}_1$ &       $:$  & $e_1 e_1 = e_2$ & $e_2 e_2=e_3$ & \\ 
\hline

$({1}){\rm BD}$  &
$\bf{N}_2$ &       $:$  & $e_1 e_1 = e_2$ & $e_2 e_1= e_3$ & $e_2 e_2=e_3$  \\ 
\end{longtable}
\noindent
Here: 
the first column indicates the varieties to which the algebra from the second column belongs
$\big($by $\mathcal{S}LA_{\delta}$ we mean that the suitable algebra does not belong to $\delta{\rm BD},$ but it belongs to $\mathcal{S}LA_{\delta}\big)$; 
the second and other columns consist of notations of algebras obtained in \cite{ks25}.
 All algebras are non-isomorphic, except 
 $\mathfrak{g}^{\alpha}_3 \cong \mathfrak{g}^{\alpha^{-1}}_3,$ 
 $\mathcal{A}_1^{\alpha} \cong \mathcal{A}_1^{\alpha^{-1}},$ and ${\mathcal N}_{6}^\alpha \cong {\mathcal N}_{6}^{-\alpha}.$

\end{document}